\documentclass[a4paper]{article}
\usepackage[T1]{fontenc}
\usepackage{lmodern}
\usepackage[english]{babel}
\usepackage[utf8]{inputenc}
\usepackage{setspace}
\usepackage{hyperref}
\usepackage{microtype}
\usepackage{quoting}
\usepackage{amsmath}
\usepackage{amssymb}
\usepackage{amsthm}
\usepackage{algorithm,algpseudocode}
\usepackage{mathrsfs}
\usepackage{booktabs}
\usepackage{caption}
\usepackage{subfig}
\usepackage{graphicx}
\usepackage{multirow}
\usepackage{yhmath}
\usepackage{xcolor}
\newtheorem{theorem}{Theorem}

\newtheorem{corollary}{Corollary}
\newtheorem{remark}{Remark}
\newcommand{\newpart}{\color{black}}

\begin{document}

\title{An Interior-Point-Inspired algorithm for Linear Programs arising in Discrete Optimal Transport}
\author{Filippo Zanetti\footnote{School of Mathematics, University of Edinburgh, Edinburgh, UK. \href{mailto:f.zanetti@sms.ed.ac.uk}{f.zanetti@sms.ed.ac.uk}} \and Jacek Gondzio\footnote{School of Mathematics, University of Edinburgh, Edinburgh, UK. \href{mailto:j.gondzio@ed.ac.uk}{j.gondzio@ed.ac.uk}}}
\date{}
\maketitle

\begin{abstract}
Discrete Optimal Transport problems give rise to very large linear
programs (LP) with a particular structure of the constraint matrix.
In this paper we present  a hybrid algorithm that mixes an
interior point method (IPM) and column generation,
specialized for the LP originating from the Kantorovich Optimal
Transport problem.
Knowing that optimal solutions of such problems display a high
degree of sparsity, we propose a column-generation-like technique
to force all intermediate iterates to be as sparse as possible.
The algorithm is implemented nearly matrix-free. 
Indeed, most of the computations avoid forming the huge matrices
involved and solve the Newton system using only a much smaller
Schur complement of the normal equations.
We prove theoretical results about the sparsity pattern of the optimal
solution, exploiting the graph structure of the underlying problem.
We use these results to mix iterative and direct linear solvers
efficiently, in a way that avoids producing preconditioners
or factorizations with excessive fill-in and at the same time
guaranteeing a low number of conjugate gradient iterations.
We compare the proposed method with two state-of-the-art solvers and show that
it can compete with the best network optimization tools in terms
of computational time and memory usage. We perform experiments with problems
reaching more than four billion variables and demonstrate the robustness
of the proposed method.
\end{abstract}


\section{Introduction}
Interior point methods (IPMs) \cite{wright} are among the most efficient optimization solvers for linear programs and are often able to outperform their competitors when it comes to very large problems. They rely on the solution of a linear system to compute the Newton direction that allows to find the next approximation and often employ iterative Krylov methods with a proper preconditioner (e.g. \cite{berga_inexact,iter:BCO-COAP,precipm,iter:GOSV,iter:OS-pcg1,iter:VOC}). They are well suited for the solution of problems with some underlying structure that allows for a simplified formulation and efficient preconditioning (e.g.\ \cite{castro,castro2,mfcs,xray}); they have been used in combination with column-generation approaches (see e.g. \cite{columngeneration}) to solve problems with many more variables than constraints; they can be formulated in a matrix-free way for solving very large problems (e.g. \cite{mfcs,mfipm}) and they have been used to solve a variety of sparse approximation problems (e.g. \cite{DecDisGonPouVio:sparse}).

Linear programs arising from discrete Optimal Transport (OT) possess many of the attractive characteristics (see e.g. \cite{compopttrans}): they can have an extremely large number of variables, but relatively few constraints and as a consequence of this, their optimal solution is extremely sparse; they have a very particular constraint matrix with a Kronecker structure which leads to a simplified formulation of the normal equations; they need to be dealt with in a matrix-free way, to avoid forming the huge, very sparse and highly structured constraint matrix.

Linear programs with many more variables than constraints are well suited to be solved using a column generation approach (see e.g. \cite{colgen,colgen3,colgen2}). In this paper, we propose a hybrid method for general purpose discrete optimal transport problems: {\newpart the proposed algorithm can be interpreted both as a very aggressive column generation approach (with just one Newton step applied to each restricted master problem) and as a very relaxed IPM (which works only with a subset of variables and ignores the other ones).} The method is highly specialized and exploits the structure of the underlying problem and the known properties of the optimal solution to efficiently mix iterative and direct solvers for the Newton linear system. It is implemented without ever forming the constraint matrix, but only accessing it via matrix-vector products, exploiting its Kronecker structure; the only matrix that is formed is the much smaller Schur complement of the normal equations. The method uses a specialized sparse linear algebra in order to tackle problems of very large dimension. 

In particular, the normal equations within the IPM are further reduced to the Schur complement and then solved either with the Conjugate Gradient (CG) method \cite{cg} or with a general sparsity-exploiting Cholesky factorization \cite{matrixcomp}. This is possible because the fill-in of the Cholesky factor of the Schur complement gets smaller and smaller when the method approaches optimality; this is rigorously proven using the graph interpretation of the OT 
problem and confirmed by extensive computational evidence. Moreover, the proposed theoretical result allows to characterize some non-trivial chordal sparsity patterns in a new way.

Many methods have been designed for solving OT problems, see e.g.\ \cite{AHA,sinkhorn,facca_benzi,shortlist,ling_okada,AHA2,shielding} and the comprehensive summary \cite{dotmark}. A similar idea to the one proposed here, where a sparse version of IPM is used to solve optimal transport problems, was recently proposed in \cite{ot_ipm}, for a particular subset of OT problems, {\newpart and in \cite{natale_todeschi} an IPM was applied to solve optimal transport problems involving finite volumes discretization.}

The strengths of the algorithm proposed in this paper, compared to the previously mentioned approaches, are: 
\begin{itemize}
\item Very general formulation, able to deal with many types of problems; indeed, the proposed method is tested on a large and varied collection of problems.
\item Adaptability to multiple cost functions, while other methods are often specialized only to one particular metric.
\item A highly specialized strategy to solve linear systems, which allows for lower and scalable requirements, both for time and memory.
\end{itemize}

{\newpart 
Among other features the proposed method heavily exploits the network structure of the constraint matrix in the problem. Several interior point algorithms have been developed for network optimization problems: in \cite{castro_lemon}, the authors derive a specialized IPM for the minimum cost flow problem on bipartite networks; the structure of the constraint matrix and normal equations is very similar to the one presented here. In \cite{resende_veiga}, another specialized IPM method is applied to the same type of networks. In \cite{frangioni_gentile}, preconditioners for IPMs applied to minimum cost flow problems are discussed. In the current paper however, the IPM is mixed with a column generation technique and the linear algebra is specialized even further, in order to exploit the extreme sparsity of the solution of the optimal transport problem.

}

The proposed method is compared with the IBM ILOG Cplex network simplex solver \cite{cplex,bertsekas,orlin}, which is a highly optimized commercial software that has been shown to be very fast and reliable when dealing with OT problems, {\newpart and with the LEMON (Library for Efficient Modelling and Optimization in Networks) network simplex solver \cite{lemon_paper}.} The computational experiments are performed on the DOTmark collection of images \cite{dotmark}, considering cost functions given by the $1-$norm, $2-$norm and $\infty-$norm; they show that the proposed method {\newpart  outperforms the Cplex solver and matches the performance of LEMON for many of the problems considered, while requiring less memory than both of them.} We performed tests with extremely large problems, with up to $4.3$ billion variables and show that the proposed method is scalable both in terms of computational time and memory requirements.

The rest of the paper is organized as follows: in Section~\ref{section_ot} we introduce the discrete optimal transport problem formulation; in Section~\ref{section_ipm} we present the hybrid interior point-column generation method in all its key features; in Section~\ref{section_ne} we analyze the structure of the normal equations and introduce the mixed iterative-direct approach for their solution; in Section~\ref{section_results} we present the test problems and show the computational results.

%%%%%%%%%%%%%%%%%%%%%%%%%%
\subsection{Notation}
We denote by $\text{vec}(\cdot)$ the operator that takes a $k_1\times k_2$ matrix and reshapes it column-wise into a $k_1k_2$ vector. We use the notation $\lceil x\rceil$ for the ceiling function, i.e.\ the smallest integer larger than or equal to $x$. We denote by $I_k$ and $\mathbf e_k$ respectively the identity matrix and the vector of all ones of size $k$. We denote by $\mathbb R^k_+$ the set of $k-$vectors with nonnegative components.

%We define the special modulo function $(x\mods k)$ as
%\[(x\mods k)=\begin{cases}
%(x\mod k) & \text{if }\,\,(x\mod k)\ne0\notag\\
%k & \text{if }\,\,(x\mod k)=0\notag
%\end{cases}.\]

%%%%%%%%%%%%%%%%%%%%%%%%%%
%%%%%%%%%%%%%%%%%%%%%%%%%%
\section{From optimal transport to optimization}
\label{section_ot}
Below we recall the Kantorovich formulation \cite{kantorovich} of the discrete Optimal Transport problem: given a starting vector $\mathbf a\in\mathbb R_+^m$ and a final vector $\mathbf b\in\mathbb R_+^n$, such that $\sum\mathbf a_j=\sum\mathbf b_j$, find a coupling matrix $\mathcal P$ inside the set
\begin{equation}\label{feasible_region}
U(\mathbf a,\mathbf b)=\Big\{\mathcal P\in\mathbb R_+^{m\times n},\ \mathcal P\mathbf e_n=\mathbf a,\ \mathcal P^T\mathbf e_m=\mathbf b\Big\}
\end{equation}
that is optimal with respect to a certain cost matrix $\mathcal C\in\mathbb R_+^{m\times n}$; i.e.\ find the solution of the following optimization problem
\begin{equation}
\label{OT_problem}
\min_{\mathcal P\in U(\mathbf a,\mathbf b)} \sum_{i,j}\mathcal C_{ij}\mathcal P_{ij}.
\end{equation}

We can interpret this OT problem as minimizing the cost of moving some mass in the configuration $\mathbf a$ into the configuration $\mathbf b$: $\mathcal C_{ij}$ gives the cost of moving a unit of mass from $\mathbf a_i$ to $\mathbf b_j$ and the optimal solution $\hat{\mathcal P}_{ij}$ tell us how much we should move from $\mathbf a_i$ to $\mathbf b_j$. The constraints given by the set $U(\mathbf a,\mathbf b)$ impose three conditions: we only move positive quantities of mass; we ensure that from each bin $i$ of configuration $\mathbf a$, we move out exactly a quantity $\mathbf a_i$ overall; we ensure that for each bin $j$ of configuration $\mathbf b$, we move in exactly a quantity $\mathbf b_j$ overall. 

In practice, the "mass" of $\mathbf a$ and $\mathbf b$ could be anything, from a probability distribution to actual physical quantities that need to be moved. If the cost matrix $\mathcal C$ is chosen appropriately (see e.g. \cite[Proposition 2.2]{compopttrans}), then the optimal solution of \eqref{OT_problem} defines a distance between $\mathbf a$ and $\mathbf b$, called the {\it q-Wasserstein} distance:
\begin{equation}
\label{wasserstein}
W_q(\mathbf a,\mathbf b)=\big(\sum_{i,j}\mathcal C_{ij}\hat{\mathcal P}_{ij}\big)^{1/q}.
\end{equation}

%%%%%%%%%%%%%%%%%%%%%%%%%%
\subsection{Kantorovich linear program}
We can rewrite the optimization problem \eqref{OT_problem} as a standard Linear Program (LP):
\begin{align}
\min_{\mathbf p\in\mathbb R^{mn}} \quad&\mathbf c^T\mathbf p\notag\\
\textup{s.t.}\quad & \begin{bmatrix}\mathbf e_n^T\otimes I_m\\I_n\otimes \mathbf e_m^T\end{bmatrix}\mathbf p=\begin{bmatrix}\mathbf a\\\mathbf b\end{bmatrix}=\colon\mathbf f\label{kantorovichlp}\\
& \mathbf p\ge0,\notag
\end{align}
where $\otimes$ denotes the Kronecker product and $\mathbf c\in\mathbb R^{mn}$ and $\mathbf p\in\mathbb R^{mn}$ are the vectorized versions of $\mathcal C$ and $\mathcal P$ respectively, $\mathbf c=\text{vec}(\mathcal C)$ and $\mathbf p=\text{vec}(\mathcal P)$.

The matrix of constraints in \eqref{kantorovichlp} has the following structure
%{\small
%\[\begin{bmatrix}
%1& & & & &1& & & & &1&\\
% &1& & & & &1& & & & &1& & & &\dots\\
% & &1& & & & &1& & & & &1&\\
% & & &1& & & & &1& & & & &1&\\
% & & & &\ddots& & & & &\ddots& & & & &\ddots\\
%\\
%1&1&1&1&\dots\\
% & & & & &1&1&1&1&\dots& & & & & &\dots\\
% & & & & & & & & & &1&1&1&1&\dots\\
% & & & & & & & &\vdots
%\end{bmatrix}.
%\]}
%If we call it
\begin{equation}
\label{Astructure}
A=\begin{bmatrix}A_1\\A_2\end{bmatrix}=\begin{bmatrix}\mathbf e_n^T\otimes I_m\\I_n\otimes \mathbf e_m^T\end{bmatrix},
\end{equation}
where $A_1$ is an operator that computes the sum of the entries of the rows of $\mathcal P$, while $A_2$ computes the sum of the entries of the columns. Notice that matrix $A$ is rank deficient by $1$. These operators can be applied in a matrix-free way. Recall the following properties of the Kronecker product \cite{horn_johnson}:
\[(Q\otimes R)^T=Q^T\otimes R^T,\]
\[(Q\otimes R)\mathbf x=\text{vec}(RXQ^T),\]
where $\text{vec}(X)=\mathbf x$.
Then, we can apply matrices $A_1$, $A_2$, $A_1^T$ and $A_2^T$ to a vector as follows:
\[A_1\mathbf x=(\mathbf e_n^T\otimes I_m)\mathbf x=\text{vec}(I_mX\mathbf e_n)=X\mathbf e_n,\]
\[A_2\mathbf x=(I_n\otimes \mathbf e_m^T)\mathbf x=\text{vec}(\mathbf e_m^TXI_n)=X^T\mathbf e_m,\]
\[A_1^T\mathbf u=(\mathbf e_n\otimes I_m)\mathbf u=\text{vec}(I_m\mathbf u\mathbf e_n^T)=\text{vec}(\mathbf u\mathbf e_n^T),\]
\[A_2^T\mathbf w=(I_n\otimes \mathbf e_m)\mathbf w=\text{vec}(\mathbf e_m\mathbf w^TI_n)=\text{vec}(\mathbf e_m\mathbf w^T),\]
where $\mathbf x\in\mathbb R^{mn}$, $X\in\mathbb R^{m\times n}$, $\mathbf u\in\mathbb R^m$, $\mathbf w\in\mathbb R^n$. Notice that the matrices involved $X$, $\mathbf u\mathbf e_n^T$ and $\mathbf e_m\mathbf w^T$ are all of dimension $m\times n$, and thus much smaller than matrix $A$.

%%%%%%%%%%%%%%%%%%%%%%%%%%
\subsection{Graph formulation}
\label{graph_formulation}
We can reformulate the optimal transport problem as a minimum cost flow problem over a bipartite graph: we have $m$ source nodes, each one with an output $\mathbf a_i$, connected to $n$ sink nodes, each requiring an input $\mathbf b_j$. The incidence matrix of such a graph is very similar to the constraint matrix $A$:
\begin{equation}
\label{incidencematrix}
\begin{bmatrix}I_n\otimes \mathbf e_m^T\\-\mathbf e_n^T\otimes I_m\end{bmatrix}.
\end{equation}
A known result of the optimal solution of an optimal transport problem is that the bipartite graph corresponding to {\newpart  any vertex of the feasible region \eqref{feasible_region}}, where the edges corresponding to zero flows have been removed, is acyclic (see e.g. \cite[Chapter 9]{network_flows} or \cite{algorithm_network}). In particular, this implies that {\newpart  there exists an optimal solution $\hat{\mathcal P}$ with at most $m+n-1$ nonzero entries.}

Figure~\ref{ot_intro} shows a small example of discrete optimal transport and the corresponding bipartite graph formulation: the problem is to move the red configuration (on the left) onto the blue configuration (on the top), where the cost is given by the physical distance between the bins of the histogram.

\begin{figure}[h]
\caption{A small example of discrete optimal transport and the corresponding bipartite graph formulation; here the intensity of the color is proportional to the quantity of mass to be moved.}
\label{ot_intro}
\centering
\includegraphics[width=.6\textwidth]{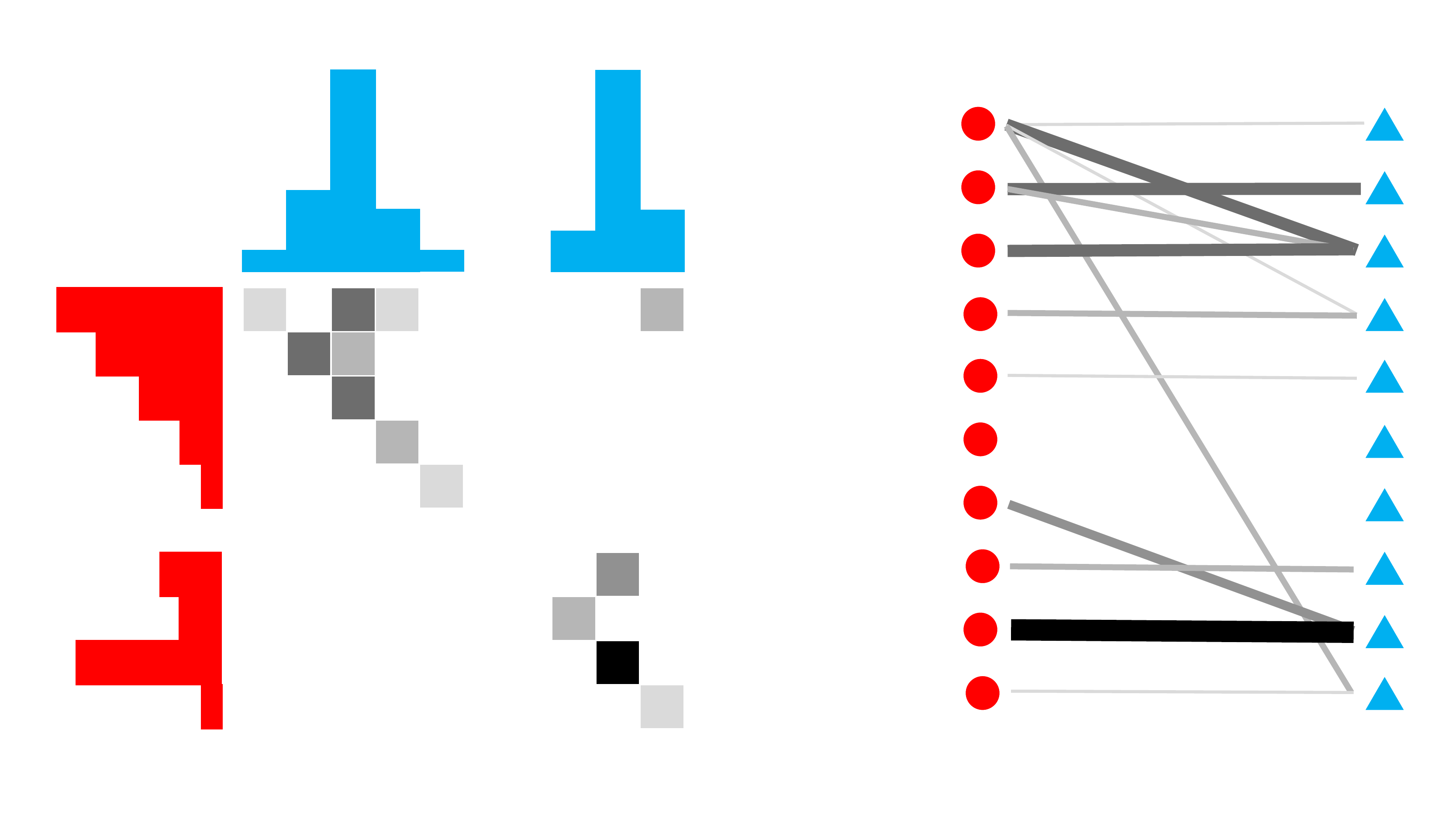}
\end{figure}

%%%%%%%%%%%%%%%%%%%%%%%%%%
%%%%%%%%%%%%%%%%%%%%%%%%%%
\section{{\newpart Interior-point-inspired algorithm for optimal transport}}
\label{section_ipm}
Problem \eqref{kantorovichlp} is a standard linear program, that can be solved using an interior point method \cite{wright}. However, the constraint matrix can be extremely large for large values of $m$ and $n$: indeed, the number of variables is $m\cdot n$ and the number of constraint is $m+n$. A linear program with such a structure of the constraint matrix is well suited to be solved by a column generation approach (see e.g.\ \cite{columngeneration,colgen,colgen3,colgen2}). We propose a method that mixes these two techniques: an interior-point-inspired method that keeps the iterates sparse at each iteration and updates the list of variables that are allowed to be non-zero in a similar way to the column generation method. We use the term \emph{support} for this list of presumed "basic" variables, to avoid confusion with the notion of \emph{basis} in the simplex and IPM communities.

{\newpart 
In the next sections, we first introduce the standard IPM formulation and then the sparsified version, highlighting the relation with  standard IPM and column generation algorithms, and the pros and cons of using such an approach.
}

%%%%%%%%%%%%%%%%%%%%%%%%%%
\subsection{Interior point method}
In an oversimplified interpretation (see \cite{wright} for more details), an interior point method solves problem \eqref{kantorovichlp} by adding a logarithmic barrier to enforce the non-negativity constraint, to obtain the Lagrangian:
\[\mathcal L(\mathbf p,\mathbf y,\mu)=\mathbf c^T\mathbf p-\mathbf y^T(A\mathbf p-\mathbf f)-\mu\sum_{i=1}^{mn}\log\mathbf p_i,\]
where {\newpart $\mathbf y\in\mathbb R^{m+n}$ is the Lagrange multiplier associated with the constraint $A\mathbf p=\mathbf f$} and the parameter $\mu$ is slowly driven to zero during the IPM iterations.

At each IPM iteration, one step of the Newton method is performed; the linear system which needs to be solved is
\begin{equation}
\label{fullsystem}
\begin{bmatrix} A\ & 0\ & 0\\0\ & A^T\ & I_{mn}\\ S\ & 0\ & P\end{bmatrix}\begin{bmatrix} \mathbf{\Delta p}\\\mathbf{\Delta y}\\\mathbf{\Delta s}\end{bmatrix}=\begin{bmatrix}\mathbf{r_1}\\\mathbf{r_2}\\\mathbf{r_3}\end{bmatrix}=\begin{bmatrix}\mathbf f-A\mathbf p\\\mathbf c-A^T\mathbf y-\mathbf s\\\sigma\mu \mathbf e_{mn}-PS\mathbf e_{mn}\end{bmatrix}.
\end{equation}
When reduced to the normal equations approach, it takes the form
\begin{equation}
\label{normalequationsipm}
A\Theta A^T\mathbf{\Delta y}=\mathbf{r_1}+A\Theta(\mathbf{r_2}-P^{-1}\mathbf{r_3})
\end{equation}
where $\Theta=\text{diag}(\theta)$ is a diagonal matrix; its entries are $\theta_j=\mathbf p_j/\mathbf s_j$, $\mathbf s$ being the dual slack variable associated with the LP \eqref{kantorovichlp}, $P=\text{diag}(\mathbf p)$ and $S=\text{diag}(\mathbf s)$, $\sigma$ is a parameter responsible for the reduction of $\mu$.

Notice that the linear system \eqref{normalequationsipm} is of dimension $(m+n)$ and thus much smaller than \eqref{fullsystem}, that is of dimension $(m+n+2mn)$; however, \eqref{normalequationsipm} contains blocks that are fully dense. Moreover, the IPM needs to work with many vectors of size $m\cdot n$, which requires excessive storage for large values of $m$ and $n$.

%%%%%%%%%%%%%%%%%%%%%%%%%%
\subsection{Sparse approach}
It has been observed (see e.g. \cite{mfipm}) that the entries of matrix $\Theta$ can be divided in two groups: for "basic" indices $j\in\mathcal B$, for which $\mathbf p_j\to\mathbf{\hat p}_j>0$ and $\mathbf s_j\to\mathbf{\hat s}_j=0$, $\theta_j$ becomes very large as the IPM approaches convergence; for the "nonbasic" indices $j\in\mathcal N$, such that $\mathbf p_j\to\mathbf{\hat p}_j=0$ and $\mathbf s_j\to\mathbf{\hat s}_j>0$, $\theta_j$ becomes very small. {\newpart  Concerning the structure of the normal equations matrix \eqref{normalequationsipm}, this separation into "basic" and "nonbasic" indices implies that the following holds close to optimality
\begin{equation}\label{normalequationsapproximation}
A\Theta A^T = \sum_{j=1}^{mn}\theta_jA_jA_j^T \approx \sum_{j\in\mathcal B}\theta_jA_jA_j^T 
\end{equation}
where $A_j$ represents the $j-$th column of $A$.
}

In the non-degenerate case, the optimal solution $\hat{\mathbf p}$ is expected to have only $m+n-1$ non-zeros, out of $m\cdot n$ entries: in the common scenario where $m=n$, this means that the density of the optimal solution decreases as $1/m$ and that at optimality most entries of $\theta$ are close to zero, and only a small fraction of them is large. 
We can exploit these properties to derive a sparsified method, where each iteration is comprised by two phases: an interior point phase on the reduced problem, and a column-generation-style update of the support.
{\newpart
Notice that, while the primal variable is expected to be sparse at optimality, the dual slack $\mathbf s$ instead is expected to be dense, due to strict complementarity; in the proposed approach, only the sparsity pattern of the primal variable $\mathbf p$ is considered. Indeed, the entries for which $\mathbf s_j\to\mathbf{\hat s}_j>0$ generate corresponding components $\theta_j = \mathbf p_j/\mathbf s_j$ that converge to zero and are therefore ignored, in order to exploit \eqref{normalequationsapproximation}.
}

In the {\it IPM phase}, we apply a standard IPM to a reduced form of the problem. {\newpart  We start from considering the subset of $\{1,2,\dots,mn\}$ that corresponds to the indices of the primal variables that are allowed to attain nonzero values; this subset, that we call \texttt{index}, contains the indices of the variables belonging to the current support. We would like to select the support such that if $j\in\texttt{index}$, then $\mathbf{\hat p}_j>0$, $\mathbf{\hat s}_j=0$, and if $j\notin\texttt{index}$, then $\mathbf{\hat p}_j=0$, $\mathbf{\hat s}_j>0$.} 

Entries $\mathbf p_j$ and $\mathbf s_j$ for $j\notin\texttt{index}$ are forced to be zero and the same applies to the components of the Newton direction, since we are not updating them; {\newpart we expect $\mathbf s_j>0$  for all $j\notin\texttt{index}$, but we set them to zero because they are ignored by the IPM phase.} We define as $\mathbf p_\text{red}$ and $\mathbf s_\text{red}$ the reduced $\mathbf p$ and $\mathbf s$ vectors that contain only the components {\newpart  corresponding to the \texttt{index} subset}. The logarithmic barrier is applied only to the set of variables in the support and the complementarity measure $\mu$ is thus computed as $\mu=(\mathbf p_\text{red}^T\mathbf s_\text{red})/\psi$, where $\psi=|\texttt{index}|$. If we call $A_\text{red}$ the submatrix of $A$ obtained considering only the columns in {\tt index}, then the linear system to solve becomes 
\begin{equation}
\label{redsystem}
\begin{bmatrix} A_\text{red}\ & 0\ & 0\\0\ & A_\text{red}^T\ & I_\psi\\ S_\text{red}\ & 0\ & P_\text{red}\end{bmatrix}\begin{bmatrix} \mathbf{\Delta p}_\text{red}\\\mathbf{\Delta y}\\\mathbf{\Delta s}_\text{red}\end{bmatrix}=\begin{bmatrix}\mathbf{r_1}\\\mathbf{r_2}_\text{red}\\\mathbf{r_3}_\text{red}\end{bmatrix}=\begin{bmatrix}\mathbf f-A_\text{red}\mathbf p_\text{red}\\\mathbf c_\text{red}-A_\text{red}^T\mathbf y-\mathbf s_\text{red}\\\sigma\mu \mathbf e_\psi-P_\text{red}S_\text{red}\mathbf e_\psi\end{bmatrix}.
\end{equation}
%The normal equations matrix takes the same form as in \eqref{normalequationsipm}; indeed $A_\text{red}\Theta_\text{red}A_\text{red}^T=A\Theta A^T$, because $\Theta$ coincides with $\Theta_\text{red}$ on the indices in {\tt index} and is zero elsewhere. However, the matrix is much sparser than before.
{\newpart 
With this approach, the diagonal of $\Theta$ is sparsified and thus the normal equations matrix is much more sparse than $A\tilde\Theta A^T$ for any possible diagonal $\tilde\Theta$ with strictly positive diagonal entries (as it would be for a standard IPM).}

In the {\it update phase}, we update the subset {\tt index} and expand or reduce the support; this can happen in two ways. The support is enlarged according to the reduced cost $\mathbf c-A^T\mathbf y$; if any component is negative, the indices corresponding to the most negative reduced costs are added to \texttt{index}. {\newpart We expect these newly added variables to satisfy $\mathbf p_j\to\mathbf{\hat p}_j>0$, $\mathbf s_j\to\mathbf{\hat s}_j=0$ and thus $\theta_j\gg1$.} The new entries are initialized with $(\mathbf p_\text{red})_j=(\mathbf s_\text{red})_j=\sqrt\mu$, in order to maintain the overall complementarity measure unaltered, {\newpart at the expense of adding a small infeasibility. This approach has been used in warm starting IPMs in \cite{warmstart} and since then has been used with success in numerous applications of IPMs which requires warm starting with newly introduced variables.}

Variables are removed from the support based on the value of $\mathbf p_\text{red}$; if $(\mathbf p_\text{red})_j$ is smaller than a certain threshold when the IPM is near convergence, we assume that {\newpart this component should satisfy
$\mathbf p_j\to\mathbf{\hat p}_j=0$, $\mathbf s_j\to\mathbf{\hat s}_j>0$ and thus $\theta_j\ll1$. Other techniques could be used to detect which variables should be removed, for example if $\theta_j$ is very small, if $\mathbf s_j$ is very large, or using some specifically developed indicators such as in \cite{IPM_indicators}.
}

In order to not perturb too much the IPM algorithm, only $m$ variables are allowed to enter or leave the support at every iteration. In practice, in the first iterations of IPM many variables enter the support, while no entry of $\mathbf p$ is small enough to leave it; in the late phase of IPM instead many variables leave the support, while almost no index is added to it. 

{\newpart 
The stopping criterion involves primal infeasibility, dual infeasibility of the current restricted problem and the complementarity measure; the method is stopped when
\[\max\Big(\frac{\|\mathbf{r_1}\|}{1+\|\mathbf f\|},\frac{\|\mathbf{r_2}_\text{red}\|}{1+\|\mathbf c_\text{red}\|},\mu\Big)<\texttt{tol},\] 
where \texttt{tol} is a predetermined tolerance. The method could also check if there are variables with negative reduced cost that should still be added to the support; however, in practice this was never the case, since the support would stabilize (and actually start to drop variables) before the IPM indicators would reach convergence.
}

The initial {\tt index} subset can be chosen in many ways according to heuristics that try to capture the sparsity pattern of the optimal solution. A simple approach is to include in the initial support the variables that correspond to a small cost $\mathbf c_j$, according to a predetermined threshold, since the optimal solution will try to allocate as much mass as possible in these low-cost variables. We know that the optimal support should include only $n+m-1$ entries, but at the beginning we choose the subset {\tt index} with a larger number of entries, usually between three and ten times more than the expected number of indices in the \texttt{index} subset corresponding to the optimal solution.

{\newpart 
%%%%%%%%%%%%%%%%%%%%%%%%%%
\subsection{Comments on the method}
In a column generation method, it is possible not to solve the restricted master problems to optimality and still converge to a solution of the master problem (see e.g.\ \cite{columngeneration}); the approach presented here can be seen as an extreme situation, where only one IPM iteration is applied to each restricted master problem before updating the support. In this way, the method does not "waste" too much time optimizing the early restricted master problems, which have a very inexact support, while it spends most of the time looking for the correct support and optimizing to full accuracy the late restricted master problems (when the support has stabilized).

The method is clearly not an interior point method, since the iterates are sparsified and do not belong to the interior of the feasible region. Such a method benefits from the sparsity of the vectors and matrices involved, at the expense of losing some dual information. Indeed, while vector $\mathbf p$ and $\theta$ are expected to be extremely sparse close to optimality, the dual variable $\mathbf s$ is not, due to the strict complementarity of linear programs. The stopping criterion indeed relies only on the dual infeasibility of the restricted problem. However, the use of a primal-dual method allows to obtain fast convergence once the support of the optimal solution has been established. An interior point method that accurately captures all the dual information would inevitably need to use fully dense vectors and normal equations matrices, which reduces substantially its applicability to large problems. The method presented here instead can be applied to huge problems, but does not reconstruct accurately the dual information during the iterations; however, the optimal dual variable $\hat{\mathbf s}$ can be computed after the algorithm terminates with one extra reduced costs computation $\hat{\mathbf s}=\mathbf c-A^T\hat{\mathbf y}$, since the Lagrange multiplier $\hat{\mathbf y}$ is reconstructed accurately.

We highlight that the IPM phase uses multiple centrality correctors \cite{colombo_gondzio} to improve the performance; a standard symmetric neighbourhood of the central path (see e.g. \cite{wright}) is used to find the correctors. Each iteration of the method however relies on a neighbourhood built with a different support. The stepsize of the method is found simply by computing the step to the boundary and scaling it down by a constant close to $1$ (e.g.\ $0.995$). 

Problems for which the support changes heavily between subsequent iterations (for example because the initial and final configuration $\mathbf a$ and $\mathbf b$ have mass concentrated on a narrow region and their nonzero patterns have small or no intersection) can create difficulties for the method, because the next IPM iterations would solve a restricted problem with potentially a very different support than the current one. These difficulties will appear for some specific problems in the numerical results section, where a large number of iterations is required before the support stabilizes.
}

%%%%%%%%%%%%%%%%%%%%%%%%%%
\subsection{Pricing method.}
\label{pricing_section}
In order to update the support, we need to compute the full reduced cost vector $\mathbf c-A^T\mathbf y$; if $m$ and $n$ are large, this can be very expensive. To reduce the cost of this operation, we propose a heuristic pricing approach. Notice the following
\[(\mathbf c-A^T\mathbf y)_j=\mathbf c_j-\mathbf y^T \mathbf a_j=\mathbf c_j-\mathbf y_{k_{1j}}-\mathbf y_{k_{2j}},\]
where $\mathbf a_j$ is the $j-$th column of $A$ and $k_{1j}$ and $k_{2j}$ {\newpart  are the source and destination nodes of edge $j$,} defined as
\[k_{1j}=\begin{cases}
(j\mod m) & \text{if }\,\,(j\mod m)\ne0\notag\\
m & \text{if }\,\,(j\mod k)=0\notag
\end{cases},\qquad k_{2j}=\Big\lceil\frac{j}{m}\Big\rceil.\]
We make the assumption that a large portion of the smallest (most negative) reduced costs correspond to small values of $\mathbf c_j$; the accuracy of this statement depends on the values attained by the Lagrange multiplier $\mathbf y$, but we observed in practice that it is true in most of the iterations. Therefore, before starting the algorithm, we can compute a subset of indices $J$, such that if $j\in J$, $\mathbf c_j<C_\text{max}$, for some fixed value $C_\text{max}$, and the corresponding indices $k_{1j}$ and $k_{2j}$. Then, during the update phase, we can quickly compute the subset of the entries of the reduced cost vector corresponding to the indices in $J$ and use only these to update the support. Of course, this method misses some of the variables to be added; thus, after a certain number of IPM iterations where we use this method, we should compute the full vector of reduced costs. In this way, we keep the low cost of a single iteration, but we likely increase the number of iterations required. The number of consecutive iterations in which we use the heuristic needs to be chosen carefully: we would like it to be large, to reduce the computational cost as much as possible, but if it is too large we risk performing useless iterations, since the next update with the full reduced costs might change drastically the support. Numerical evidence suggests that the reduced costs should be refreshed every 3 or 4 iterations.

%Algorithm~\ref{sparseipm} summarizes the method just described.

%\renewcommand{\thealgorithm}{SparseIPM}
%\begin{algorithm}[h]
%\caption{}
%\label{sparseipm}
%\small
%\textbf{Input:} $m$, $n$, $\mathbf f=[\mathbf a^T\ \ \mathbf b^T]^T$, $\mathbf c$, {\tt index}
%
%\textbf{Initialize} $\mathbf p_\text{red}^0$, $\mathbf y^0$, $\mathbf s_\text{red}^0$
%\begin{algorithmic}[1]
%\While{$\max\Big(\frac{\|\mathbf{r_1}\|}{1+\|\mathbf f\|},\frac{\|\mathbf{r_2}_\text{red}\|}{1+\|\mathbf c_\text{red}\|},\mu\Big)>\texttt{tol}$}
%\State $\mathbf{r_1}=\mathbf f-A_\text{red}\mathbf p_\text{red}$
%\State $\mathbf{r_2}_\text{red}=\mathbf c_\text{red}-A_\text{red}^T\mathbf y-\mathbf s_\text{red}$
%\State $\mu=(\mathbf p_\text{red}^T\mathbf s_\text{red})/\psi$
%\State Choose $\sigma$ and compute $\mathbf{r_3}_\text{red}=\sigma\mu\mathbf e_\psi-P_\text{red}S_\text{red}\mathbf e_\psi$
%\State Compute Newton direction $(\mathbf{\Delta p}_\text{red},\mathbf{\Delta y},\mathbf{\Delta s}_\text{red})$ by solving \eqref{redsystem}
%\State Compute stepsizes $\alpha_p$ and $\alpha_s$ 
%\State $\mathbf p_\text{red}\leftarrow\mathbf p_\text{red}+\alpha_p\mathbf{\Delta p}_\text{red}$
%\State $\mathbf y\leftarrow\mathbf y+\alpha_s\mathbf{\Delta y}$
%\State $\mathbf s_\text{red}\leftarrow\mathbf s_\text{red}+\alpha_s\mathbf{\Delta s}_\text{red}$
%\State Perform pricing, with heuristic or full reduced costs
%\State Update {\tt index}, $\mathbf p_\text{red}$, $\mathbf s_\text{red}$
%\EndWhile
%\end{algorithmic}
%\end{algorithm}

%%%%%%%%%%%%%%%%%%%%%%%%%%
\subsection{Structure of the normal equations matrix}
\label{NormEqProp}
The normal equations matrix in \eqref{normalequationsipm} takes the form
\begin{equation}\label{normaleq}
A\Theta A^T=\begin{bmatrix}M & V\\V^T & N\end{bmatrix}\end{equation}
where $V\in\mathbb R^{m\times n}$, $\text{vec}(V)=\theta$, $M\in\mathbb R^{m\times m}$ and $N\in\mathbb R^{n\times n}$ are diagonal and
\[M_{ii}=\sum_{t=0}^{n-1} \theta_{i+tm},\,\, i=1,\dots,m,\qquad N_{jj}=\sum_{t=1}^{m} \theta_{(j-1)m+t},\,\, j=1,\dots,n.\]
{\newpart The reader can verify this claim by direct computation using matrix \eqref{Astructure}. This structure is the same as found in \cite{castro_lemon}. }

Notice the meaning of matrices $M$ and $N$: 
\begin{equation}\label{matrix_relations}
V\mathbf e_n=M\mathbf e_m,\qquad V^T\mathbf e_m=N\mathbf e_n.\end{equation}
i.e.\ $M_{kk}$ is the sum of the entries of row $k$ of $V$, while $N_{kk}$ is the sum of the entries of column $k$ of $V$. 

In a standard IPM, matrix $V$ would be completely dense, since the entries $\theta_j$ are all strictly positive. However, in the proposed hybrid method most entries of $\theta$ are exactly zero, making matrix $V$ extremely sparse.

%%%%%%%%%%%%%%%%%%%%%%%%%%
%%%%%%%%%%%%%%%%%%%%%%%%%%
\section{Solution of the normal equations}
\label{section_ne}
To solve a linear system involving matrix \eqref{normaleq}, we can use the Schur complement approach: the solution of
\[\begin{bmatrix}M & V\\V^T & N\end{bmatrix}\begin{bmatrix}\alpha_1\\\alpha_2\end{bmatrix}=\begin{bmatrix}\beta_1\\\beta_2\end{bmatrix}\]
is given by either
\[\alpha_1=S_N^{-1}(\beta_1-VM^T\beta_2),\quad\alpha_2=N^{-1}(\beta_2-V^T\alpha_1),\]
or
\[\alpha_2=S_M^{-1}(\beta_2-V^TM^{-1}\beta_1),\quad\alpha_1=M^{-1}(\beta_1-V\alpha_2),\]
where $S_N=M-VN^{-1}V^T$ and $S_M=N-V^TM^{-1}V$ are the two Schur complements. Such an approach is convenient in this case since matrices $M$ and $N$ are diagonal and thus very easy to invert. Notice that both Schur complements are rank deficient by 1, but this is not a problem in practice: both iterative methods and direct solvers can deal with singular matrices by solving the corresponding least squares problem.

Let us analyze $S_M$: if we sum the rows of matrix $V^TM^{-1}V$ we obtain, exploiting \eqref{matrix_relations}:
\[V^TM^{-1}V\mathbf e_n=V^TM^{-1}M\mathbf e_m=V^T\mathbf e_m=N\mathbf e_n.\]

The Schur complement $S_M$ is equal to matrix $N$ minus the previous matrix. Therefore
\begin{align}
|(S_M)_{kk}|-\sum_{i\ne k}|(S_M)_{ki}|&=|N_{kk}-(V^TM^{-1}V)_{kk}|-\sum_{i\ne k}|(V^TM^{-1}V)_{ki}|\notag\\
&=N_{kk}-(V^TM^{-1}V)_{kk}-\sum_{i\ne k}(V^TM^{-1}V)_{ki}\notag\\
&=N_{kk}-\sum_{i=1}^n(V^TM^{-1}V)_{ki}\notag\\
&=N_{kk}-(V^TM^{-1}V\mathbf e_n)_k=0.\label{diagdom}
\end{align}
This means that the Schur complement $S_M$ is weakly diagonally dominant. Similarly, the same can be said about $S_N$.

%If we consider the other Schur complement $S_N=M-VN^{-1}V^T$, then
%\[VN^{-1}V^T\mathbf e_m=VN^{-1}N\mathbf e_n=V\mathbf e_n=M\mathbf e_m\]
%and we conclude similarly that
%\begin{equation}
%\label{diagdom2}
%|(S_N)_{kk}|-\sum_{i\ne k}|(S_N)_{ki}|=0.
%\end{equation}

%%%%%%%%%%%%%%%%%%%%%%%%%%
\subsection{Sparsity pattern of the Schur complement}
We use results from graph theory to study the sparsity pattern of the Schur complement. Given a set of nodes $\mathcal N$ and edges $\mathcal E$, we call {\it adjacency matrix} of the undirected graph $\mathcal G(\mathcal N,\mathcal E)$ a symmetric matrix $\mathcal A$ such that $\mathcal A_{ij}\ne0$ if there exists an edge between nodes $i$ and $j$.

An undirected graph is {\it chordal} if every cycle of length greater than three has a chord, i.e. an edge between two non-consecutive nodes in the cycle. We say that the sparsity pattern of a matrix is chordal if the matrix can be interpreted as the adjacency matrix of a chordal graph. The following Theorem relates chordal graphs and positive definite matrices.
\begin{theorem}
A sparsity pattern $\mathcal S$ is chordal if and only if for every positive definite matrix with sparsity pattern $\mathcal S$ there exists a symmetric permutation of its rows and columns that produces a Cholesky factor with zero fill-in.
\end{theorem}
\begin{proof}
See \cite[Theorem 9.1]{chordal}.
\end{proof}

An undirected bipartite graph has an adjacency matrix of the kind
\begin{equation}
\label{bipartite_adjacency}
\mathcal A=\begin{bmatrix}0 & \mathcal M\\ \mathcal M^T & 0\end{bmatrix}
\end{equation}
where $\mathcal M$ is called the {\it biadjacency matrix}. Equivalently, given any matrix $\mathcal M$, we can build a corresponding bipartite graph for which $\mathcal M$ is its biadjacency matrix.

The following result is important for the matrices considered in this paper:
\begin{theorem}
\label{theorem_graph}
{\newpart Consider a matrix $\mathcal M\in\mathbb R^{m\times n}$, and the corresponding bipartite graph for which $\mathcal M$ is the biadjacency matrix; if every cycle of length larger than or equal to $8$ in the bipartite graph has at least a chord, then $\mathcal M\mathcal M^T$ has chordal sparsity pattern.}
\end{theorem}
\begin{proof}
Let us call $\mathcal N_0$ and $\mathcal N_0'$ the two groups of nodes in the bipartite graph; we call this graph the primary graph and indicate it as $\mathcal{PG}(\mathcal N_0,\mathcal N_0';\mathcal E_0)$. 

Since it is a bipartite graph, we can construct its adjacency matrix as in \eqref{bipartite_adjacency}.
We notice that $\mathcal A^2$ has two nonzero blocks
\[\mathcal A^2=\begin{bmatrix}\mathcal M\mathcal M^T & 0\\0 & \mathcal M^T\mathcal M\end{bmatrix}.\]
A known fact is that $\mathcal A^2$ tells us which nodes are connected by paths of length two; in this case, it is possible only if both nodes belong to $\mathcal N_0$ or to $\mathcal N_0'$, thus explaining the structure of $\mathcal A^2$. 

Let us focus on the first of the two blocks $\mathcal M\mathcal M^T$. We can build a secondary graph using only the nodes $\mathcal N_0$; two nodes are connected by an edge in the secondary graph if there is a path of length two between them in the primary graph. We denote the secondary graph as $\mathcal{SG}(\mathcal N_0,\mathcal E_2)$, where the subscript $2$ indicates that the edges correspond to 2-paths in $\mathcal{PG}$. Notice that $\mathcal M\mathcal M^T$ has the same sparsity pattern as the adjacency matrix of $\mathcal{SG}$. {\newpart  We want to prove that if every cycle of length at least $8$ in $\mathcal{PG}$ has a chord, then $\mathcal{SG}$ is chordal.}

Figure \ref{graphs} shows an example of a primary graph and the corresponding secondary graph. We have not represented self-loops, which are always present in $\mathcal{SG}$ by construction. %In the following pictures, we use dotted lines to indicate the primary graph and dashed lines for the secondary graph.

\begin{figure}[h]
\caption{}
\label{}
\centering
\subfloat[Primary graph and associated secondary graph.\label{graphs}]{\includegraphics[width=.4\textwidth]{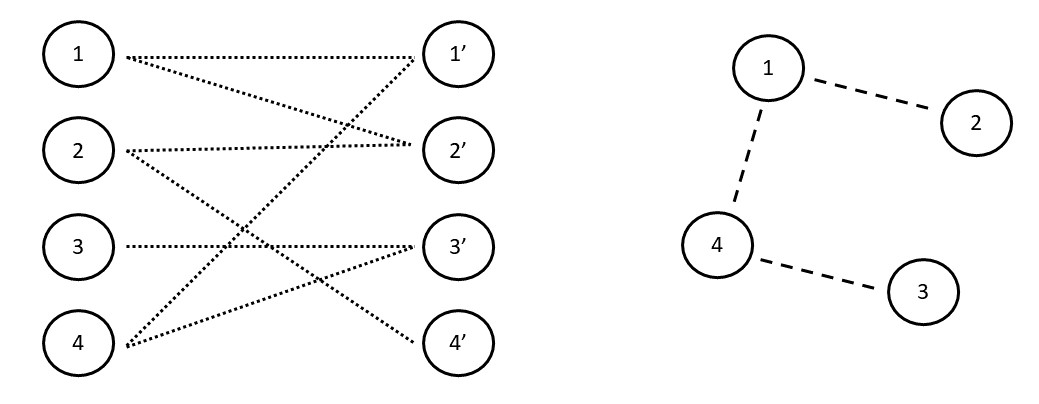}}\hspace{20pt}
\subfloat[Construction of a cycle in $\mathcal{SG}$.\label{loop3}]{\includegraphics[width=.4\textwidth]{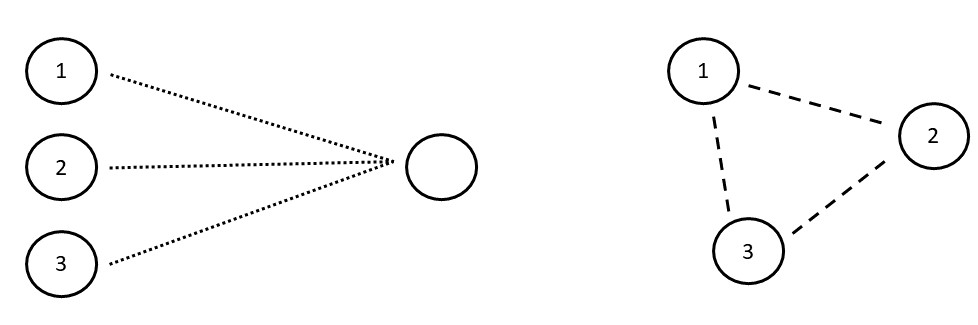}}
\end{figure}
%\vspace{-5pt}

%\begin{figure}[h]
%\caption{Primary graph $\mathcal{PG}$ and associated secondary graph $\mathcal{SG}$.}
%\label{graphs}
%\centering
%\includegraphics[width=.45\textwidth]{images/graphs}
%\end{figure}

Let us analyze the possible cycles in the secondary graph. A trivial way to have a cycle in $\mathcal{SG}$ is if a subset of nodes in $\mathcal N_0$ are all connected to the same node in $\mathcal N_0'$, in the graph $\mathcal{PG}$. Figure \ref{loop3} shows an example of this. In particular, any subset of nodes in $\mathcal N_0$ that are all connected to the same node in $\mathcal N_0'$ forms a complete subgraph of $\mathcal{SG}$.

%\begin{figure}[h]
%\caption{Construction of a cycle in $\mathcal{SG}$.}
%\label{loop3}
%\centering
%\includegraphics[width=.45\textwidth]{images/loop3}
%\end{figure}

We call primary cycle and secondary cycle respectively a cycle in $\mathcal{PG}$ and in $\mathcal{SG}$. Similarly, we define primary edges and secondary edges.

Let us argue by contradiction: let us try to construct a chordless secondary cycle of length at least 4 . Let us suppose that the cycle is formed by a subset of nodes $\mathcal Q\subseteq\mathcal N_0$. We know that if there are three or more nodes in $\mathcal Q$ connected to the same node in $\mathcal N_0'$ in the primary graph, then there exists a chord in the secondary cycle; thus, let us assume that this does not happen. This means that we need an {\it auxiliary} node in $\mathcal N_0'$ to connect in the primary graph each couple of subsequent nodes in the secondary cycle. Suppose that the chordless cycle in the secondary graph is $q_1\rightarrow q_2\rightarrow\dots\rightarrow q_k\rightarrow q_1$, with $k\ge4$. To form the secondary edge between $q_1$ and $q_2$, we need a path of length two in the primary graph which involves a node $q_1'\in\mathcal N_0'$. To form the secondary edge between $q_2$ and $q_3$, we need a path of length two in the primary graph which involves a node $q_2'\in\mathcal N_0'$, since we cannot use $q_1'$ again or we would produce a chord. Continuing this reasoning, to have a chordless cycle we need to use a different auxiliary node $q_j'\in\mathcal N_0'$ for each pair of subsequent nodes in the secondary cycle. This implies that any edge in $\mathcal{SG}$ corresponds to two edges in $\mathcal{PG}$. To construct the chordless cycle, we would need $k$ auxiliary nodes in $\mathcal N_0'$ and $2k$ edges in the primary graph. {\newpart Since $k\ge4$, this implies that there must be a cycle of length at least $8$ in $\mathcal{PG}$; this cycle is chordless because we are assuming that there are no three edges starting from the same node in $\mathcal N_0'$. This is a contradiction.}

\end{proof}

{\newpart Notice that, if we apply the Theorem to matrix $\mathcal M^T$, it follows that also $\mathcal M^T\mathcal M$ has a chordal sparsity pattern. Notice also that the Theorem applies, as a particular case, to matrices $\mathcal M$ that correspond to acyclic bipartite graphs.}

We will use this Theorem to analyze the sparsity pattern of the Schur complements obtained using the optimal solution $(\hat{\mathbf p},\hat{\mathbf y},\hat{\mathbf s})$; however, when the algorithm approaches optimality and the complementarity products get close to zero, the entries of $\Theta$ tend to zero or infinity. For the sake of formalism, define the matrix $\mathcal V$ as the sparsity pattern of $V$; then during the IPM iterations, $\mathcal V\to\hat{\mathcal V}$, where
\[\begin{cases}
\hat{\mathcal V}_{ij}=1\quad\text{if}\quad\hat{\mathcal P}_{ij}>0\notag\\
\hat{\mathcal V}_{ij}=0\quad\text{if}\quad\hat{\mathcal P}_{ij}=0\notag
\end{cases},\]
where $\text{vec}(\hat{\mathcal P})=\hat{\mathbf p}$.
At a certain IPM iteration, the Schur complements have the same sparsity pattern as $VV^T$ or $V^TV$; therefore, getting close to optimality, they tend to have sparsity pattern $\hat{\mathcal V}\hat{\mathcal V}^T$ or $\hat{\mathcal V}^T\hat{\mathcal V}$. The next result shows that these are chordal matrices.

\begin{corollary}
\label{corollarychordal}
Matrices $\hat{\mathcal V}\hat{\mathcal V}^T$ and $\hat{\mathcal V}^T\hat{\mathcal V}$ have chordal sparsity patterns.
\end{corollary}
\begin{proof}
Recall the graph formulation presented in Section~\ref{graph_formulation}: $\hat{\mathcal P}$ is the biadjacency matrix of the bipartite graph corresponding to the optimal solution, which is known to be acyclic. By construction, $\hat{\mathcal V}$ and $\hat{\mathcal P}$ have the same sparsity pattern. Therefore, matrix $\hat{\mathcal V}$ satisfies the assumption of Theorem~\ref{theorem_graph}.
{\newpart 
Notice that it there is more than one optimal solution, at least one of them has to correspond to an acyclic bipartite graph; therefore, the corollary holds at least for this specific solution $\mathbf{\hat p}$.
}
\end{proof}

%%%%%%%%%%%%%%%%%%%%%%%%%%
\subsection{Mixing direct and iterative solvers}
We consider two options for solving system \eqref{normalequationsipm} with matrix \eqref{normaleq}: direct approach using Cholesky factorization \cite{matrixcomp} and iterative one using preconditioned conjugate gradient \cite{cg}.

We can identify two stages of the IPM. In the early iterations, the support is far from the optimal one, so we expect that if we compute the Cholesky factorization of the Schur complement, there would be a lot of fill-in; in this phase though, $\Theta$ is well conditioned and we can expect a simple preconditioner like incomplete Cholesky to work well. In the late iterations instead, we know that the normal equations matrix becomes extremely ill-conditioned due to the behaviour of matrix $\Theta$, which makes it difficult to find a good preconditioner for the conjugate gradient method; however, computing a full Cholesky factorization of the Schur complement would be extremely cheap since both $S_M$ and $S_N$ get closer and closer to a sparse chordal matrix, as was showed in Corollary~\ref{corollarychordal}.

The proposed approach first applies the conjugate gradient with incomplete Cholesky as preconditioner, with a value of the drop tolerance that is lowered every time too many linear iterations are performed; then, based on a switching criterion, when we decide that the Schur complement is "close enough" to the optimal one, we switch to a direct solution using an exact factorization with approximate minimum degree ordering. 

The Schur complements are weakly diagonally dominant, as shown in \eqref{diagdom} and this guarantees that the incomplete factorization never breaks down in exact arithmetic; however, since the matrices are only weakly diagonally dominant, we may need to lift their diagonal with a small perturbation, in order to prevent from numerical inaccuracies making the matrices lose diagonal dominance property.

When using a full factorization, we employ the $LDL^T$ algorithm; since both Schur complements are singular, we expect one entry of matrix $D$ to be zero. We deal with this problem adding a small shift to the diagonal of $D$, since the entry that should be zero could become negative due to numerical errors, {\newpart and using the Matlab backslash operator to apply matrices $L$ and $D$. Notice that this issue could also be dealt with using a modified Cholesky factorization that substitutes small pivots with an infinitely large value in the computation of the factors.}

\begin{remark}
{\newpart 
{\it
Notice that in this paper the plain Cholesky factorization from Matlab was employed; however, to achieve maximum efficiency, it is possible to update a symbolic factorization, that keeps track of the sparsity pattern of the factors, every time the support is updated.
}}
\end{remark}

%%%%%%%%%%%%%%%%%%%%%%%%%%
\subsection{Switching strategy}
In order to switch between the iterative solver and the direct factorization, we need to be sure that the level of fill-in in the Cholesky factor is not going to be too large, or otherwise the method may lose efficiency. Matrix $V$ which is used to build the Schur complement is the biadjacency matrix of a bipartite graph that is not perfectly acyclic (this happens only at optimality), but gets closer and closer to being so; we can expect the number of cycles found in the graph to decrease and correspondingly the fill-in level of the Cholesky factor of the Schur complements to be reduced when the IPM gets close to optimality. One strategy to switch between iterative and direct solver can be based on the number of cycles that are found in the bipartite graph associated with matrix $V$ at any IPM iteration; however, counting the number of cycles in the graph can be quite expensive. A simpler approach is to count the number of edges in the bipartite graph to decide when to switch: intuitively, the more edges there are, the more likely it is to find cycles; moreover, if the number of edges is decreasing fast, it means that the IPM is getting close to optimality and thus we can expect the Cholesky factor to be sufficiently sparse. Notice that each edge corresponds to a variable that is allowed to be nonzero; the number of edges is thus readily available, since it is the number of variables in the support. Therefore, we can switch from iterative to direct solver as soon as we detect that the number of variables in the support is decreasing fast enough: to do so, we compute its relative variation in the last 5 iterations and switch method as soon as this number is sufficiently large.

\begin{remark}
{\it
{\newpart 
Notice that it is possible to employ other switching strategies, for example based on the value of the IPM parameter $\mu$; however, there is little difference in the final results and the current strategy, based on the speed of edge removal, is particularly simple to tune.
Notice also that, if the number of variables in the support happens to increase after the method has switched to the full factorization, it may be necessary to switch back to an iterative solution of the linear system. This problem never occurred in the experiments presented here, but such a strategy may be required for solving other problems.
}
}
\end{remark}

%%%%%%%%%%%%%%%%%%%%%%%%%%
%%%%%%%%%%%%%%%%%%%%%%%%%%
\section{Numerical results}
\label{section_results}

\vspace{10pt}
%%%%%%%%%%%%%%%%%%%%%%%%%%
\subsection{Test problems}
We tested the proposed algorithm on the {\it DOTmark} (Discrete Optimal Transport benchmark) collection of problems \cite{dotmark}. It includes 10 classes of images, each containing 10 images; the images come from simulations based on various probability distributions (class 1-7), geometric shapes (class 8), classic test images (class 9) and scientific observations using microscopy (class 10). Figure \ref{dotmark_images} shows one example from each class.

\begin{figure}[h]
\caption{Example of images from each class of the DOTmark collection.}
\label{dotmark_images}
\captionsetup[subfigure]{labelformat=empty}
\centering
\subfloat[Class 1]{\includegraphics[width=.15\textwidth]{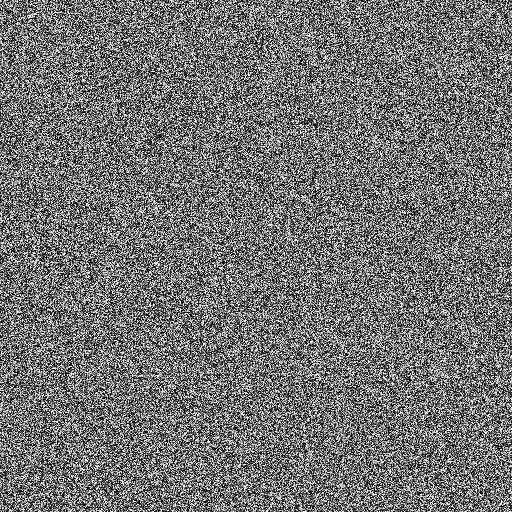}}
\subfloat[Class 2]{\includegraphics[width=.15\textwidth]{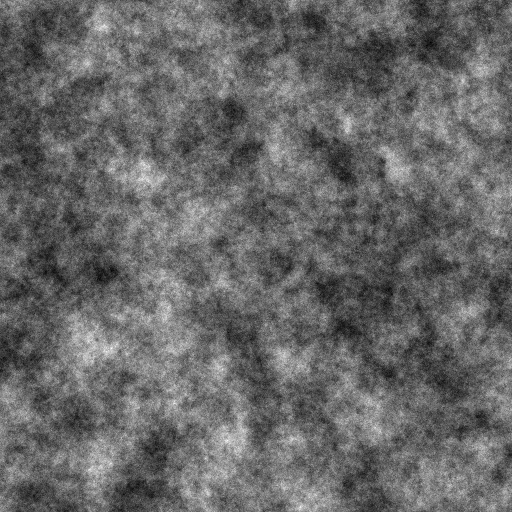}}
\subfloat[Class 3]{\includegraphics[width=.15\textwidth]{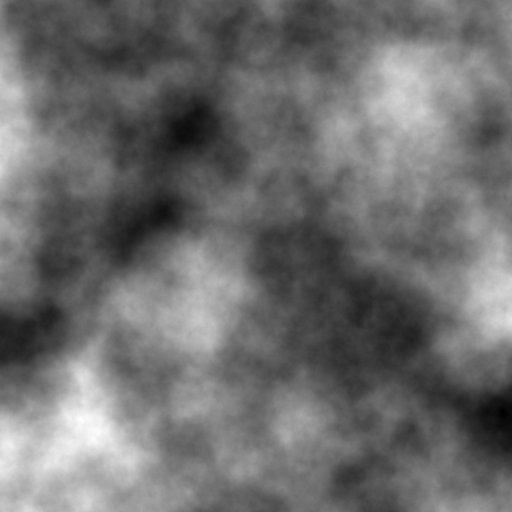}}
\subfloat[Class 4]{\includegraphics[width=.15\textwidth]{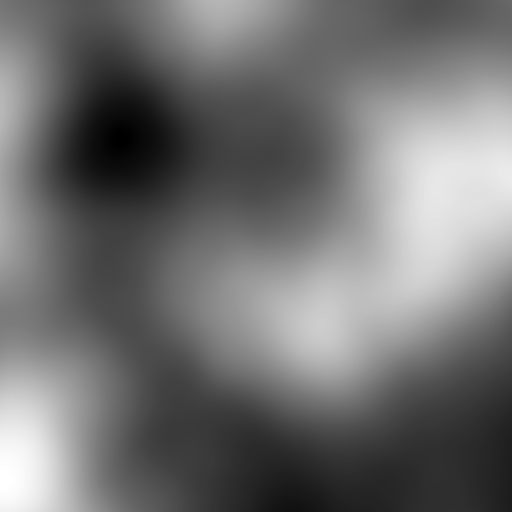}}
\subfloat[Class 5]{\includegraphics[width=.15\textwidth]{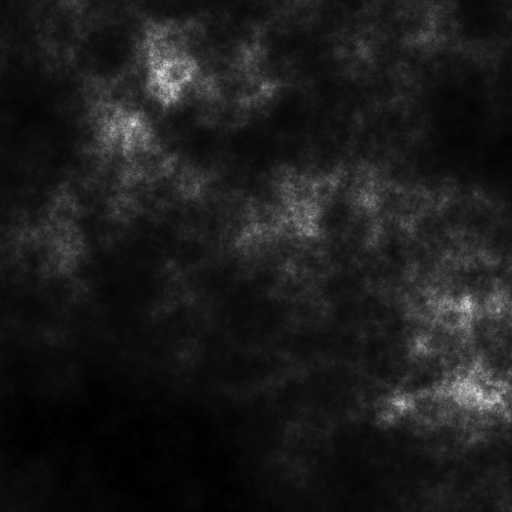}}\\
\subfloat[Class 6]{\includegraphics[width=.15\textwidth]{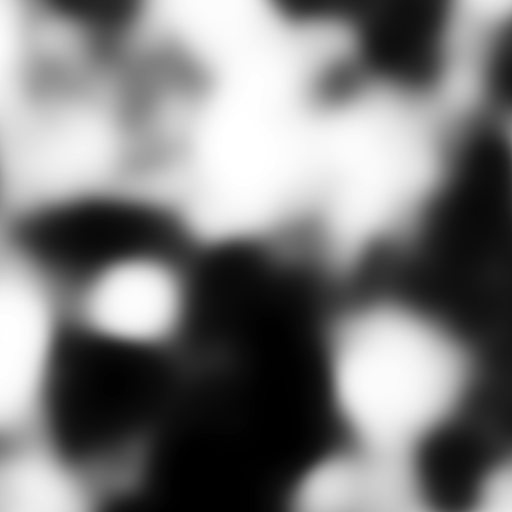}}
\subfloat[Class 7]{\includegraphics[width=.15\textwidth]{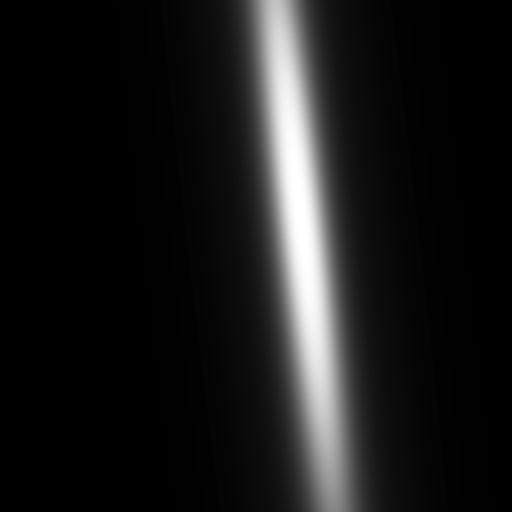}}
\subfloat[Class 8]{\includegraphics[width=.15\textwidth]{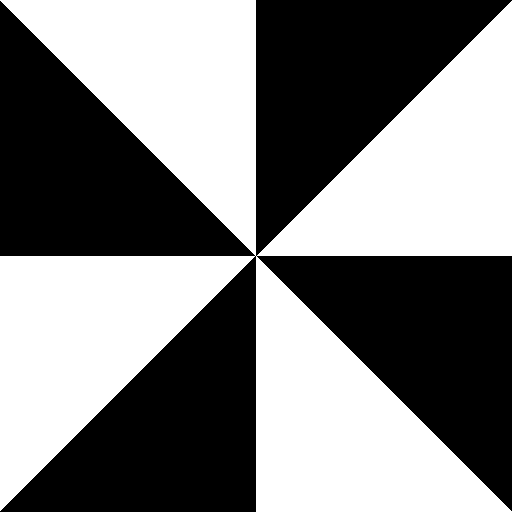}}
\subfloat[Class 9]{\includegraphics[width=.15\textwidth]{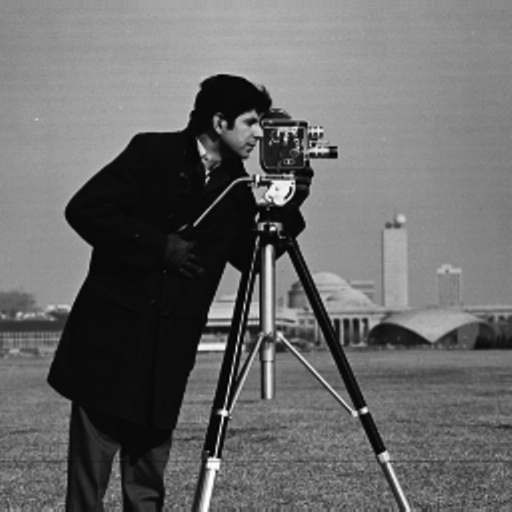}}
\subfloat[Class 10]{\includegraphics[width=.15\textwidth]{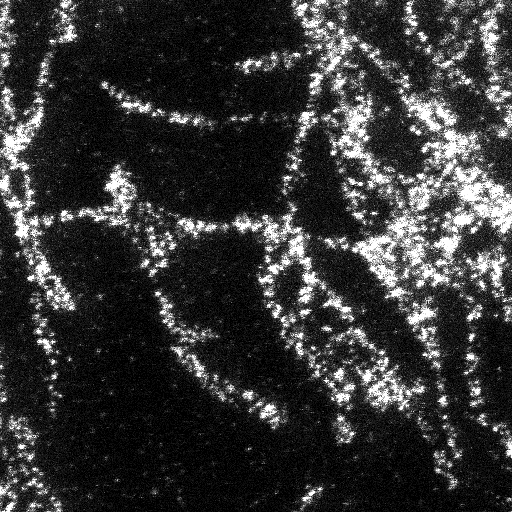}}
\end{figure}

Within a certain class, we can solve the OT problem between any couple of images, giving 45 problems per each class and a total of 450 instances overall. The images are available at different resolutions: we considered mainly the case $32\times32$, $64\times64$ and $128\times128$ pixels and show partial results also with $256\times 256$. The problems that arise for a certain resolution \texttt{res} have a number of constraints equal to $2\,\texttt{res}^2$ and a number of variables $\texttt{res}^4$; Table~\ref{problemsize} shows the dimension of the problems considered.

\begin{table}[h]
\small
\caption{Size of the problems solved}
\label{problemsize}
\centering
\begin{tabular}{rrr}
\hline
{\tt res} & Constraints & Variables ($\times 10^6)$\\
\hline
32 & 2,048 & 1.0\\
64 & 8,192 & 16.8\\
128 & 32,768 & 268.4\\
256 & 131,072 & 4,295.0\\
\hline
\end{tabular}
\end{table}

As cost function, we consider the $1-$distance, $2-$distance and $\infty-$distance: the vectors $\mathbf a$ and $\mathbf b$ represent the vectorized forms of two images $\mathcal A$ and $\mathcal B$; $\mathcal C_{ij}$ is the cost of moving mass from position $i$ to position $j$; in the images, position $i$ corresponds to the entry $\mathcal A_{\alpha_i,\beta_i}$, while position $j$ corresponds to $\mathcal B_{\alpha_j,\beta_j}$. Then, the distances used are
\[\mathcal C^1_{ij}=|\alpha_i-\alpha_j|+|\beta_i-\beta_j|,\]
\[\mathcal C^2_{ij}=\sqrt{(\alpha_i-\alpha_j)^2+(\beta_i-\beta_j)^2},\]
\[\mathcal C^\infty_{ij}=\max(|\alpha_i-\alpha_j|,|\beta_i-\beta_j|).\]

Notice the following: if we consider an image with $k$ pixels per side, the maximum possible distance between two pixels is approximately $2k$, $\sqrt{2}k$ or $k$, respectively for the $1$, $2$ and $\infty$ distance; all of these values grow linearly with $k$. Therefore, the parameter $C_\text{max}$ in Section~\ref{pricing_section} should also scale linearly with the size of the image considered: in this way, the pricing heuristics considers always the same fraction of the total number of variables for each resolution considered.

%%%%%%%%%%%%%%%%%%%%%%%%%%
\subsection{Solvers}

There are many methods that have been proposed for the solution of the Kantorovich linear program \eqref{kantorovichlp} (see e.g.\ \cite{AHA,sinkhorn,shortlist,AHA2,shielding}); a comparison of them can be found in \cite{dotmark}. {\newpart  We compared the proposed method to two very efficient implementations of the network simplex algorithm \cite{bertsekas,orlin}, coming from IBM Cplex \cite{cplex} and the LEMON (Library for Efficient Modelling and Optimization in Networks) library \cite{lemon,lemon_paper}. 

Cplex is a highly optimized commercial software that showed excellent reliability and consistency of computational times throughout all the 10 classes of the DOTmark collection in \cite{dotmark}. To use it directly from its callable library, the optimal transport problems are generated in C and the network structure, in terms of nodes and edges, is passed to a Cplex CPXNET model. Only the network optimizer of Cplex is considered in this work, since the standard dual simplex and the barrier solver were approximately 10 to 20 times slower, for resolutions 32 and 64 (this fact was also noticed in \cite{dotmark}).

LEMON \cite{lemon_paper} is a library for network optimization written in C$++$ that has shown to outperform Cplex in some applications \cite{castro_lemon}. Out of the four algorithms available in LEMON for the minimum cost flow problem (network simplex, cost scaling, capacity scaling and cycle cancelling), the network simplex was the method which gave the best results for the problems considered in this paper. To use LEMON, the OT problem is generated directly in C$++$ in terms of nodes and edges lists, that are then passed to a LEMON network simplex model. However, since LEMON only accepts integer input data for a network simplex model, we had to convert the OT problem data to integer form (multiplying them by a large constant) before performing the computation. This approach may lead to a loss of accuracy and potentially to integer overflow if larger problems are considered. The model was run with the pivot rule parameter set to {\tt block-search}.

The interior point method instead is implemented in MATLAB, in a way that exploits as much as possible the built-in functions, to reduce the computational time required.} 
The parameters for the IPM are: feasibility and optimality tolerance $10^{-6}$; conjugate gradient tolerance for predictors $10^{-6}$ and correctors $10^{-3}$ (we used a different tolerance for predictors and correctors, as was shown in \cite{cg_termination}); maximum number of IPM iterations $200$; maximum number of conjugate gradient iterations (for each call) $1000$; maximum number of correctors $3$.

{\newpart 
All the experiments were performed on the  University of Edinburgh School of Mathematics computing server, which is equipped with four 3.3GHz octa-core Intel Gold 6234 processors and 500GB of RAM. The specific versions of the software used were as follows: MATLAB R2018a, Cplex 20.1, Lemon 1.3.1 and GCC 4.8.5 as compiler.

\begin{remark}
{\it The Lemon library could not be installed on the computing server, but we could only use it by compiling the source code directly. The performance of Lemon relative to the other solvers, when installed on a local machine, may be slightly better that the one shown here.}
\end{remark}
}

%%%%%%%%%%%%%%%%%%%%%%%%%%
\subsection{Results}
{\newpart  We report the results for the whole DOTmark collection, for all three solvers, for resolutions 32, 64 and 128 pixels.}
\begin{remark}
{\it The IPM is an inexact method, while the network simplex finds the exact solution. {\newpart  The accuracy of the approximate solution can be measured with the} relative Wasserstein error (RWE):
\[RWE(\mathbf a,\mathbf b)=\Big|\frac{W_2^\text{IPM}(\mathbf a,\mathbf b)-W_2^\text{Cplex}(\mathbf a,\mathbf b)}{W_2^\text{Cplex}(\mathbf a,\mathbf b)}\Big|\]
where the 2-Wasserstein distance is defined in \eqref{wasserstein}. 
{\newpart  In all the experiments performed, the RWE was of the order of $10^{-6}-10^{-8}$, which indicates that the IPM solution was very accurate.}}
\end{remark}

Table~\ref{table_results} reports the average results for each class, for resolutions $32$, $64$ and $128$ and for each cost function considered, in terms of iterations ({\tt iter}) and time ({\tt time}) of the proposed method, Cplex time ({\tt Cplex}) and Lemon time ({\tt Lemon}). Table~\ref{table_details} instead reports the average over all the three cost functions for each class of the number of CG iterations performed (\texttt{CG}), the maximum fill-in percentage level of the factorization of the Schur complement (\texttt{fill}) and the number of iterative (\texttt{iter}) and direct (\texttt{dir}) iterations performed. Figure~\ref{perf_profile} shows the performance profiles of the computational time, while Figure~\ref{results_128} compares the time taken by the three solvers for each problem at resolution $128$.

\begin{table}[h]
\caption{Average results for each class and cost function}
\label{table_results}
\centering
\tiny
\begin{tabular}{c|c|rrrr|rrrr|rrrr}
\hline
&& \multicolumn{4}{c}{$32\times32$}&\multicolumn{4}{c}{$64\times64$}&\multicolumn{4}{c}{$128\times128$}\\
\cline{3-6}
\cline{7-10}
\cline{11-14}
Dist & Class & {\tt iter} & {\tt time} & {\tt Cplex} & {\tt Lemon} & {\tt iter} & {\tt time} & {\tt Cplex} & {\tt Lemon} & {\tt iter} & {\tt time} & {\tt Cplex} & {\tt Lemon} \\
\hline
\multirow{10}{*}{1} & 1 & 10.9 & 0.30 & 0.41 & 0.18 & 13.6 & 2.12 & 12.02 & 7.83 & 19.1 & 31.43 & 1,220.27 & 219.01\\ 
& 2 & 11.5 & 0.35 & 0.36 & 0.18 & 18.0 & 3.84 & 11.17 & 7.25 & 35.4 & 108.06 & 1,140.72 & 216.28\\ 
& 3 & 16.0 & 0.58 & 0.36 & 0.16 & 26.9 & 7.79 & 10.75 & 6.63 & 44.4 & 161.26 & 1,100.77 & 159.85\\ 
& 4 & 20.3 & 0.84 & 0.34 & 0.17 & 38.9 & 15.69 & 10.77 & 6.42 & 68.0 & 279.05 & 1,097.77 & 145.91\\ 
& 5 & 25.1 & 1.11 & 0.35 & 0.16 & 40.1 & 18.93 & 12.11 & 7.04 & 139.3 & 1,107.91 & 1,234.57 & 143.99\\ 
& 6 & 18.7 & 0.60 & 0.40 & 0.17 & 36.0 & 16.88 & 13.25 & 7.54 & 57.4 & 227.49 & 1,355.14 & 217.01\\ 
& 7 & 30.9 & 1.42 & 0.29 & 0.20 & 68.2 & 40.14 & 12.86 & 7.30 & 83.3 & 1,074.91 & 1,306.70 & 154.39\\ 
& 8 & 17.2 & 0.58 & 0.33 & 0.16 & 50.2 & 51.89 & 11.66 & 6.55 & 74.8 & 729.58 & 1,189.68 & 155.40\\ 
& 9 & 14.8 & 0.45 & 0.33 & 0.15 & 24.4 & 7.20 & 11.67 & 7.17 & 49.8 & 155.11 & 1,196.78 & 172.44\\ 
& 10 & 22.2 & 0.79 & 0.37 & 0.18 & 41.2 & 18.30 & 10.80 & 6.41 & 60.1 & 246.45 & 1,107.03 & 156.48\\ 
\hline
\multirow{10}{*}{2} & 1 & 18.5 & 0.69 & 0.51 & 0.18 & 26.2 & 5.16 & 13.48 & 7.57 & 30.4 & 45.01 & 1,139.28 & 190.94\\ 
& 2 & 29.7 & 1.18 & 0.50 & 0.19 & 72.1 & 16.90 & 16.15 & 9.00 & 114.0 & 481.26 & 1,369.31 & 446.40\\ 
& 3 & 49.8 & 2.22 & 0.53 & 0.18 & 90.6 & 30.47 & 18.66 & 9.41 & 150.9 & 814.64 & 1,583.77 & 600.02\\ 
& 4 & 60.0 & 2.93 & 0.54 & 0.20 & 96.4 & 40.41 & 20.17 & 9.66 & 166.8 & 720.83 & 1,715.39 & 615.64\\ 
& 5 & 56.2 & 2.50 & 0.56 & 0.20 & 108.6 & 43.24 & 21.77 & 9.69 & 181.7 & 2,628.16 & 1,846.61 & 502.28\\ 
& 6 & 46.6 & 2.14 & 0.56 & 0.19 & 86.5 & 32.06 & 21.54 & 10.23 & 159.3 & 4,613.47 & 1,829.04 & 705.84\\ 
& 7 & 63.6 & 3.55 & 0.56 & 0.23 & 106.3 & 56.12 & 21.38 & 9.40 & 183.1 & 902.35 & 1,819.55 & 332.04\\ 
& 8 & 34.4 & 1.48 & 0.48 & 0.17 & 70.8 & 29.91 & 18.51 & 8.70 & 127.8 & 424.25 & 1,573.27 & 488.31\\ 
& 9 & 43.8 & 1.94 & 0.53 & 0.18 & 99.7 & 30.85 & 18.57 & 10.33 & 145.5 & 501.34 & 1,577.44 & 644.97\\ 
& 10 & 50.4 & 2.17 & 0.55 & 0.21 & 87.1 & 28.17 & 16.49 & 7.00 & 154.2 & 1,425.74 & 1,394.50 & 273.86\\ 
\hline
\multirow{10}{*}{$\infty$} & 1 & 11.7 & 0.48 & 0.46 & 0.18 & 15.0 & 3.57 & 12.67 & 7.06 & 19.1 & 36.52 & 945.47 & 199.09\\ 
& 2 & 12.2 & 0.50 & 0.42 & 0.17 & 17.6 & 6.09 & 12.28 & 6.74 & 34.2 & 119.78 & 923.97 & 156.30\\ 
& 3 & 15.6 & 0.75 & 0.36 & 0.15 & 25.2 & 10.33 & 12.35 & 6.40 & 48.4 & 210.68 & 928.24 & 133.26\\ 
& 4 & 19.1 & 0.98 & 0.37 & 0.16 & 35.6 & 16.12 & 12.63 & 6.21 & 63.8 & 292.34 & 948.83 & 132.83\\ 
& 5 & 23.7 & 1.31 & 0.38 & 0.16 & 36.4 & 16.56 & 12.62 & 6.52 & 84.4 & 467.41 & 949.53 & 127.95\\ 
& 6 & 18.7 & 0.91 & 0.43 & 0.16 & 30.2 & 13.85 & 14.47 & 6.86 & 61.4 & 312.69 & 1,087.24 & 163.88\\ 
& 7 & 28.3 & 1.69 & 0.40 & 0.18 & 50.5 & 25.33 & 13.00 & 6.41 & 98.7 & 1,663.26 & 966.95 & 128.30\\ 
& 8 & 17.0 & 0.96 & 0.38 & 0.15 & 33.1 & 36.00 & 11.08 & 6.31 & 62.1 & 664.18 & 829.04 & 151.12\\ 
& 9 & 15.0 & 0.67 & 0.42 & 0.15 & 23.9 & 9.46 & 12.11 & 6.60 & 49.2 & 212.13 & 907.62 & 139.94\\ 
& 10 & 20.7 & 1.07 & 0.42 & 0.17 & 35.6 & 19.29 & 11.60 & 6.34 & 59.4 & 277.50 & 872.06 & 129.68\\ 
\hline
\end{tabular}
\end{table}

\begin{table}[h]
\caption{Details of the method for each class}
\label{table_details}
\centering
\scriptsize
\begin{tabular}{c|rrrr|rrrr|rrrr}
\hline
& \multicolumn{4}{c}{$32\times32$}&\multicolumn{4}{c}{$64\times64$}&\multicolumn{4}{c}{$128\times128$}\\
\cline{2-5}
\cline{6-9} 
\cline{10-13}
Class & \texttt{CG} & \texttt{fill} & \texttt{iter} & \texttt{dir} & \texttt{CG} & \texttt{fill} & \texttt{iter} & \texttt{dir} & \texttt{CG} & \texttt{fill} & \texttt{iter} & \texttt{dir} \\
\hline
1 & 1,354.7 & 9.6 & 13.2 & 0.4 & 2,697.6 & 3.9 & 17.0 & 1.3 & 3,916.1 & 1.5 & 21.1 & 1.7\\
2 & 1,342.2 & 11.7 & 14.9 & 2.9 & 3,460.2 & 4.3 & 27.9 & 8.0 & 10,925.3 & 1.6 & 74.3 & 2.1\\
3 & 1,840.6 & 12.2 & 23.8 & 3.3 & 4,564.0 & 2.8 & 43.5 & 4.1 & 15,844.7 & 1.9 & 104.0 & 2.6\\
4 & 2,144.5 & 11.7 & 29.9 & 3.2 & 5,130.7 & 3.0 & 53.7 & 3.2 & 16,666.0 & 1.5 & 107.2 & 5.1\\
5 & 1,670.1 & 12.9 & 28.9 & 6.1 & 4,765.0 & 3.0 & 51.1 & 10.6 & 20,031.6 & 2.6 & 168.8 & 4.5\\
6 & 1,976.5 & 11.4 & 25.8 & 2.3 & 5,046.1 & 3.2 & 48.6 & 2.3 & 14,258.0 & 2.9 & 95.0 & 2.9\\
7 & 2,498.0 & 11.7 & 38.7 & 2.3 & 6,230.1 & 3.3 & 72.3 & 2.7 & 19,147.9 & 2.4 & 161.8 & 6.3\\
8 & 1,521.1 & 10.9 & 21.7 & 1.2 & 6,544.9 & 3.6 & 49.4 & 2.0 & 13,875.7 & 1.6 & 90.2 & 3.9\\
9 & 1,798.9 & 11.5 & 22.2 & 2.4 & 4,647.1 & 3.0 & 43.5 & 5.9 & 15,194.9 & 1.9 & 94.9 & 2.6\\
10 & 1,728.7 & 13.3 & 26.2 & 4.9 & 5,211.4 & 4.7 & 44.8 & 9.9 & 14,045.0 & 1.8 & 97.9 & 17.1\\
\hline
\end{tabular}
\end{table}

{\newpart  The reported computational time accounts only for the optimization phase. However, there is some pre-processing time that Cplex and Lemon need to take in order to prepare the network model. In particular, after generating the problem (in terms of the vectors $\mathbf a$, $\mathbf b$ and $\mathbf c$), some arrays are created which contain the nodes of the graph and the respective supply, and the edges and the respective cost. Then, the network model for the respective solver needs to be created and the graph information needs to be passed to the model. These two phases, which are not required by the hybrid IP solver, in the larger instance of resolution $128$ require overall approximately 5 to 10 seconds. This extra time needed by both Cplex and LEMON is not included in the table and in the performance profiles. Moreover, the creation of the full network requires a substantial amount of memory, that is avoided by the method proposed here.}

\begin{remark}
{\it The reader should keep in mind that the number of iterations presented refers to the overall process of adding/removing variables and optimizing; it is thus not a surprise that there are instances where the iteration count is higher than what would be expected from a standard IPM. The cost of a single iteration however is considerably lower, given the high degree of sparsity of the vectors and matrices involved.}
\end{remark}

\begin{figure}[h]
\caption{Performance profiles of the computational time for the whole DOTmark collection, for the three cost functions, at different resolution.}
\label{perf_profile}
\captionsetup[subfigure]{labelformat=empty}
\centering
\subfloat[Resolution $32\times32$]{\includegraphics[width=.6\textwidth]{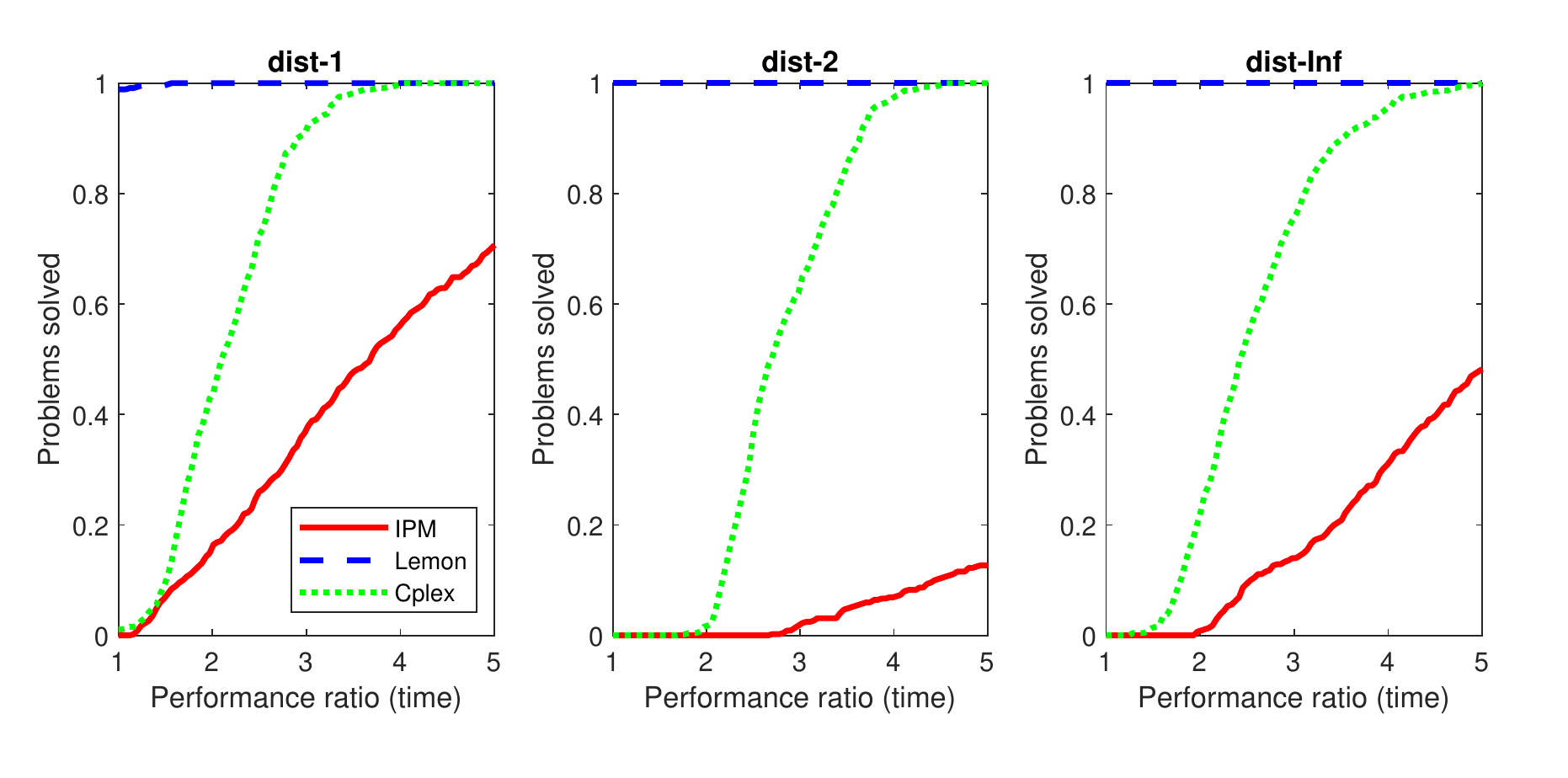}}\\
\subfloat[Resolution $64\times64$]{\includegraphics[width=.6\textwidth]{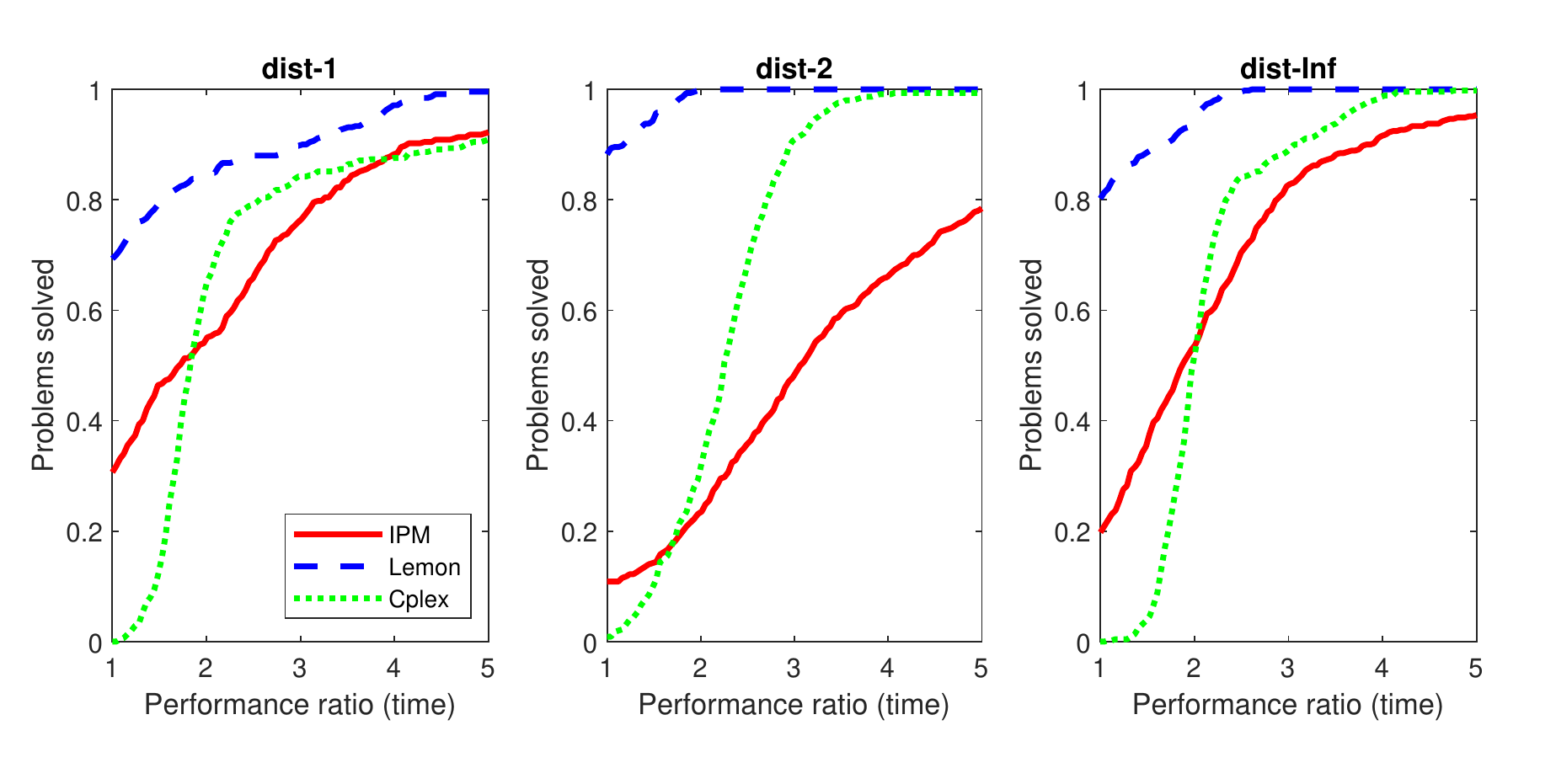}}\\
\subfloat[Resolution $128\times128$]{\includegraphics[width=.6\textwidth]{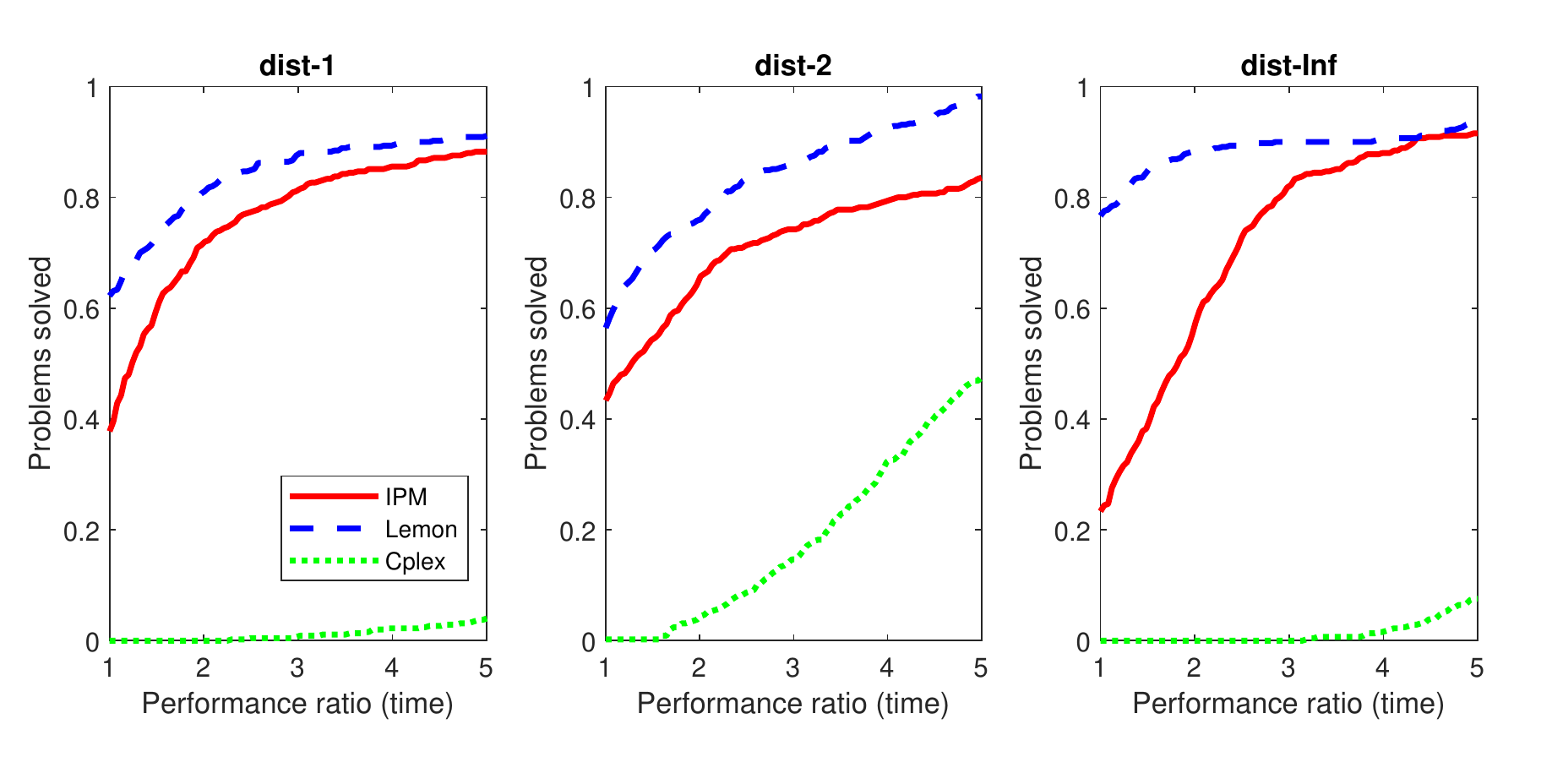}}
\end{figure}

\begin{figure}[h]
\caption{Computational time taken by the three solvers for each problem in the DOTmark collection, grouped by class and distance, for resolution $128$.}
\label{results_128}
\captionsetup[subfigure]{labelformat=empty}
\centering
\subfloat[Dist $1$]{\includegraphics[width=.35\textwidth]{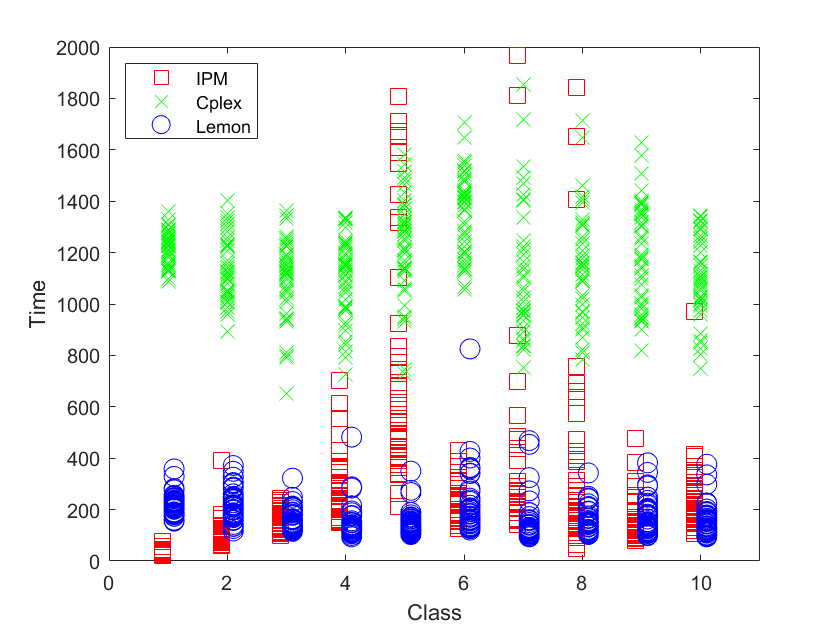}}
\subfloat[Dist $2$]{\includegraphics[width=.35\textwidth]{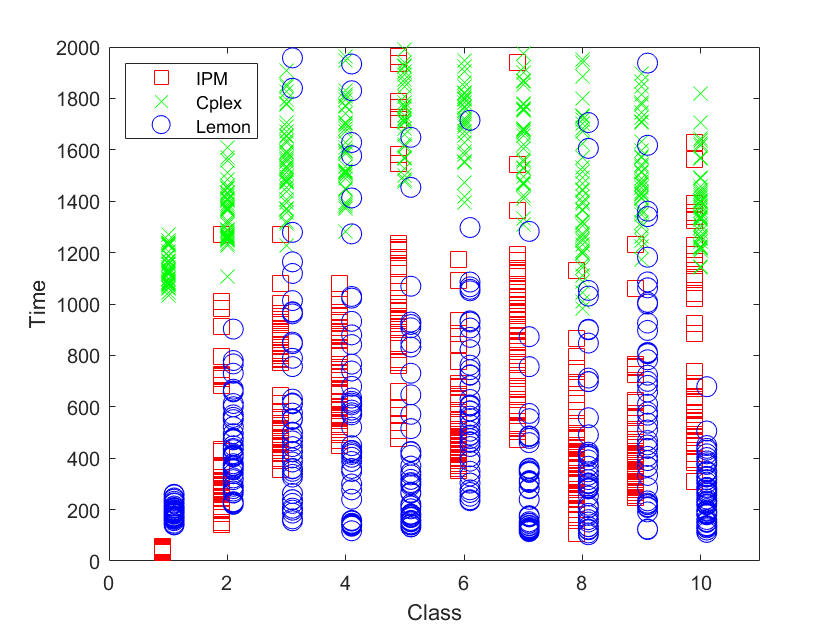}}\\
\subfloat[Dist $\infty$]{\includegraphics[width=.35\textwidth]{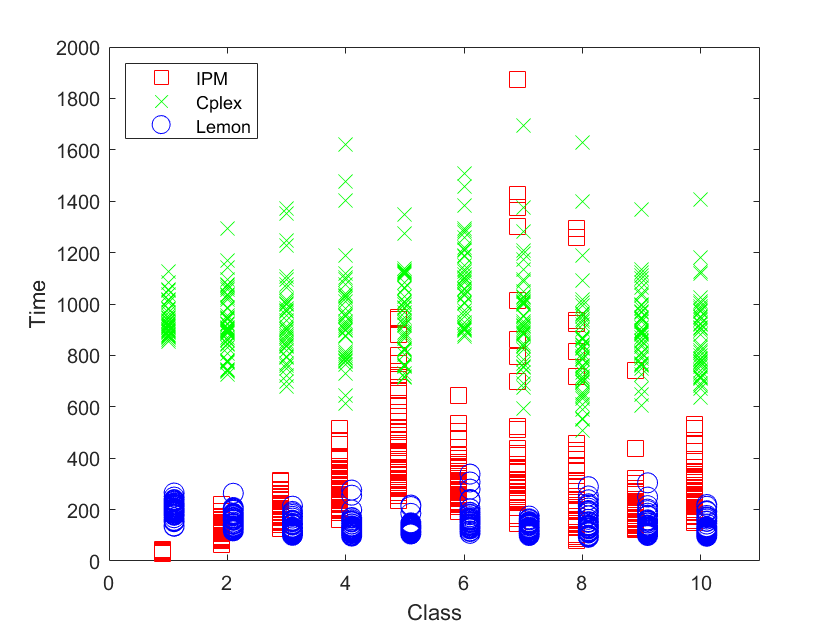}}
\end{figure}

{\newpart  
From these results, it is clear that for small problems ($32\times32$) Lemon outperforms both Cplex and the proposed method; however, the IPM scales better and becomes competitive for larger instances. Cplex instead scales badly and is considerably slower than the other solvers at resolution $128$. Given that both Cplex and Lemon are highly optimized libraries written in a compiled language, while the method proposed here is implemented in Matlab, we consider these results very promising. 

From Figure~\ref{results_128}, we can see that the proposed method behaves particularly well for problems in classes $1$, $2$, $3$, $9$ and $10$, since the time taken is close to the one taken by Lemon, and overall the spread of the times for all the problems in these classes remains narrow. Classes $5$, $7$ and $8$ instead are the ones with the weaker results: the computational time is considerably larger than the one taken by Lemon, with a large variability among the instances of these classes.} Observing the images in Figure \ref{dotmark_images}, it seems that the method performs better when the mass of the image is distributed evenly everywhere, while it struggles when it is concentrated in a narrow region. {\newpart  It is also evident that all three solvers struggle more when using the $2-$distance as cost function.}

From Table~\ref{table_details} we can see that the maximum fill-in level is independent of the class of the images and gets smaller when the problem becomes larger, since the density of the optimal solution decreases as $1/\texttt{res}^2$. We can also see that on average only the very last IPM iterations employ a direct factorization, which is what we wanted to achieve with the switching criterion proposed.

%%%%%%%%%%%%%%%%%%%%%%%%%%
\subsection{Results for large instances}
{\newpart 
When considering problem with resolution $256\times256$, both Cplex and Lemon crash: they required more than 400GB of memory, in large part needed to store the huge network structure. The IPM however does not need to store the network and thus scales better in terms of memory, requiring at most approximately 120GB of memory; indeed, the only large data that the IPM stores are the cost vector $\mathbf c$, the reduced costs and the factorization of the Schur complement. Given the huge size of the problem, we show partial results for classes $1$, $2$ and $3$ and distances $1$ and $\infty$ (due to the long computational time required by the $2-$norm cost function at resolution 256).

\begin{figure}[h]
\caption{Logarithmic plot of the average computational time against the number of variables for classes $1$, $2$ and $3$ and cost functions $1$ and $\infty$.}
\label{res256}
\centering
\subfloat[$1-$distance]{\includegraphics[width=.7\textwidth]{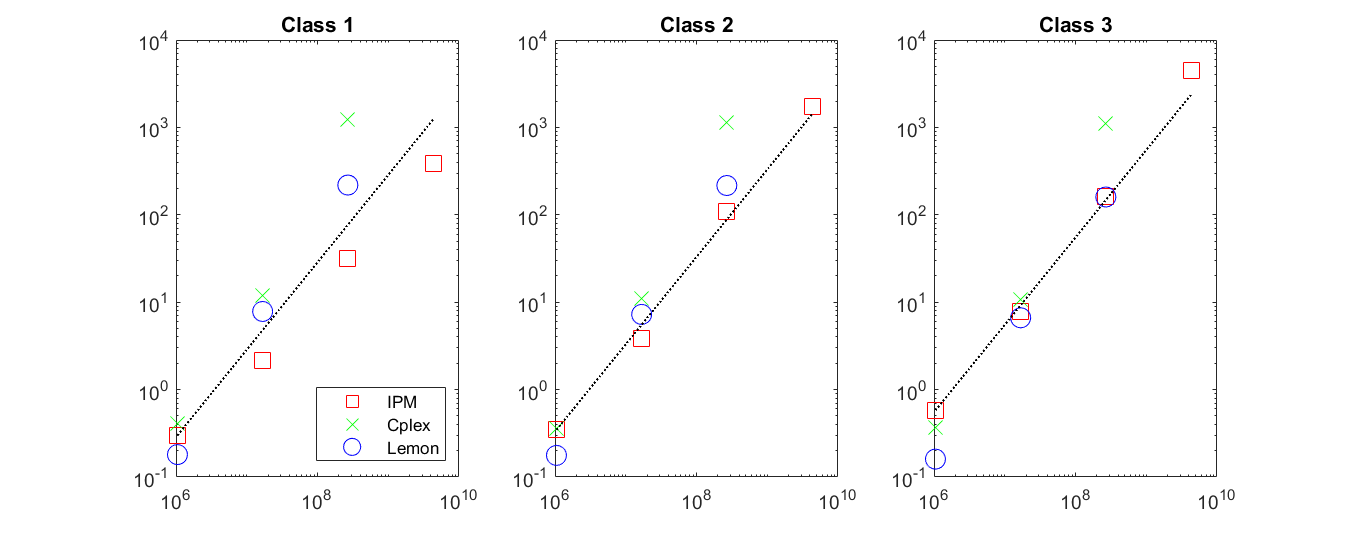}}\\
\subfloat[$\infty-$distance]{\includegraphics[width=.7\textwidth]{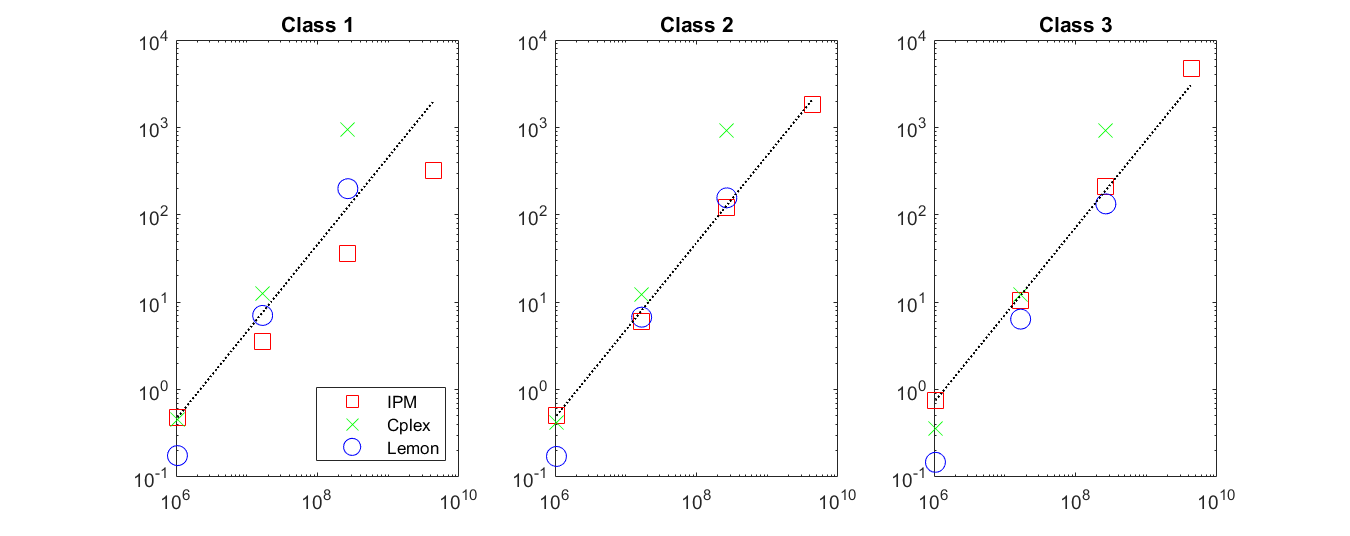}}
\end{figure}

Figure~\ref{res256} shows the growth of the average computational time, in a logarithmic plot, for the three solvers, from resolution $32$ to $128$ (for Cplex and Lemon) or $256$ (for the IPM), for the three classes considered; the dotted line shows a linear growth for comparison. 

This Figure highlights that the computational time of the proposed method grows linearly with the number of variables, while the other two solvers scale super-linearly. The growth of the computational time of the IPM appears to be sub-linear in some cases, for instance for the problems in class $1$. A possible explanation is the following: while the number of variables in the optimization problem grows by a factor $16$ (proportional to $m\cdot n$) when the resolution doubles, the expected number of nonzeros in the solution grows only by a factor $4$ (proportional to $m+n$); therefore, the number of variables kept in the support only needs to grow by a factor of $4$ to capture the nonzeros of the solution. As a consequence, the cost of vector operations grows proportionally to the square root of the number of variables and similarly for the number of nonzeros in the Schur complement and Cholesky factor (indeed, the Cholesky decomposition becomes relatively sparser, as shown in Table~\ref{table_details}, under the columns \texttt{fill}). However, some operations require the use of the large vector $\mathbf c$, like the computation of the full reduced costs; it is not surprising then that, if the number of IPM and conjugate gradient iterations does not grow excessively when changing resolution (as in the case of class $1$), the computational time growth is sub-linear.
}

{\newpart 
The numerical results have shown that the method proposed scales better both in terms of computational time and memory requirements, despite being still a prototype code written in MATLAB. The scalable results presented for problems with up to $4$ billion variables prove the efficacy of the proposed column-generation-inspired IPM for optimal transport problems and potentially for other optimization problems with a large network structure.

}

%%%%%%%%%%%%%%%%%%%%%%%%%%
%%%%%%%%%%%%%%%%%%%%%%%%%%
\section{Conclusion}
In this paper, a {\newpart  hybrid method that mixes interior point and column generation algorithms} was introduced for solving linear programs arising from discrete optimal transport. A careful analysis of the matrices involved allows for an efficient way of solving the Newton linear systems that arise within the interior point phase. Experimental results show that the proposed approach is able to outperform {\newpart  the network simplex solver of Cplex and compete with the highly efficient library LEMON, in particular in terms of memory required}. Results with huge scale problems, with up to four billion variables, confirm the robustness of the method.

Further research can potentially improve the performance even more, in particular in relation to:
\begin{itemize}
\item The choice of the starting subset of nonzero variables, which could be obtained by some more sophisticated heuristics, based on the initial and final distribution of mass; a better starting point could allow for a smaller size of the support, which would lead to faster computations.
\item The switching strategy from iterative method to direct factorization, which could reduce even further the maximum fill-in level of the Cholesky factor and/or the number of overall conjugate gradient iterations.
\item The pricing algorithm.
\end{itemize}

Moreover, the entropic regularized OT problem can be trivially dealt with using an interior point method, just by adding regularization to matrix $\Theta$; however, the structure of the optimal solution changes and the considerations about the linear solver made in this paper need to be reconsidered for that specific case. 

{\newpart 
Further research is also needed to deal with the convergence properties of the proposed method. Indeed, the combination with column generation prevents the method from keeping the known polynomial complexity of IPMs and the inexact solution of the restricted master problems, which are terminated after merely one IPM iteration, makes it difficult to apply standard results about convergence of column generation techniques. Theoretical results on convergence may also give an explanation to the relatively poorer behaviour of the method when the mass of the initial and final distribution is highly concentrated.
}

\section*{Acknowledgements}
The authors are grateful to the two anonymous referees, whose comments helped us to produce a stronger and more precise version of the paper.

\bibliographystyle{siam}
\bibliography{biblio}

\end{document}